\documentclass[a4paper,leqno]{article}
\usepackage[a4paper,centering,vscale=0.77,hscale=0.77]{geometry}\usepackage{amsmath,amsthm,amssymb,bbm,bm}
\usepackage{xcolor}
\usepackage{verbatim}
\usepackage{todonotes}
\usepackage{mathrsfs}

\usepackage{thmtools}
\usepackage[pdfdisplaydoctitle,colorlinks,breaklinks,urlcolor=blue,linkcolor=blue,citecolor=blue]{hyperref}
\usepackage{enumitem}

\newtheorem{theorem}{Theorem}[section]
\newtheorem{lemma}[theorem]{Lemma}
\newtheorem{proposition}[theorem]{Proposition}
\newtheorem{corollary}[theorem]{Corollary}
\newtheorem{definition}[theorem]{Definition}
\newtheorem{remark}[theorem]{Remark}
\newtheorem{example}[theorem]{Example}
\newtheorem{Assumption}[theorem]{Assumption}
\newtheorem{notation}[theorem]{Notation}%\newenvironment{proof}[1][Proof]{\textbf{#1.} }{\ \mathbb Rule{0.5em}{0.5em}}

\begin{document}

\title{Uniqueness of the invariant measure  and asymptotic stability for the 2D Navier-Stokes equations with multiplicative noise}
\author{Benedetta Ferrario\thanks{Dipartimento di Scienze Economiche e Aziendali, Universit\`a di Pavia, 27100 Pavia, Italy.
E-mail: \texttt{benedetta.ferrario@unipv.it}}
\and Margherita Zanella\thanks{Department of Mathematics, Politecnico di Milano,
Via E.~Bonardi 9, 20133 Milano, Italy.
E-mail: \texttt{margherita.zanella@polimi.it}}}

\maketitle

\begin{abstract}
We establish the uniqueness and  the asymptotic stability of the invariant measure for the two-dimensional Navier-Stokes equations driven by a multiplicative noise which is either bounded or  with a sublinear or a linear growth. We work on an “effectively elliptic” setting, that is we require that the range of the covariance operator contains the unstable directions. % (low modes).
We exploit the generalized asymptotic coupling techniques of \cite{GHMR17} and \cite{KS}, used by these authors for the stochastic Navier-Stokes equations with additive noise. 
Here  we show how these methods are flexible enough to deal with multiplicative noise as well. 
A crucial role in our argument is played by the Foias-Prodi estimate in expected valued, which has
a different form (exponential or polynomial decay) according to the growth condition of the multiplicative noise.
\end{abstract}

\noindent
{\textbf{Keywords:} Two dimensional stochastic Navier-Stokes equations, multiplicative noise, invariant measure, generalized coupling method, mixing, Foias–Prodi estimate in expected value.}
\\
{\bf MSC}:  
35Q30,  %  	Navier-Stokes equations 
35R60, %   	PDEs with randomness, stochastic partial differential equations 
60H30, %   	Applications of stochastic analysis (to PDEs, etc.)
60G10, %  	Stationary stochastic processes
60H15. %   	Stochastic partial differential equations (aspects of stochastic analysis)

\tableofcontents

\section{Introduction}

In the last decades there have been a large number of papers on the subject of ergodicity for stochastic partial differential equations (SPDEs), 
see for instance \cite{CIME} \cite{KS_BOOK}, \cite{noi}, \cite{noi_2}, \cite{Sca_Zan} and the references therein. The large majority of the works concerns SPDEs driven by an additive stochastic forcing term, whereas the papers dealing with multiplicative-type noises are much scarcer.

In \cite{GHMR17} Glatt-Holtz, Mattingly and Richards 
 identify an intuitive and conceptually simple framework for proving the uniqueness of the invariant measure by a generalized asymptotic coupling technique. 
This approach has been developed in many other papers; we refer also  to \cite{HMS11} by   Hairer, Mattingly   and Scheutzow 
 and to  \cite{BKS} by  Butkovsky, Kulik and Scheutzow. 
In \cite{GHMR17} many examples of  PDEs driven by an additive noise are considered, for which this framework led to streamlined proofs of uniqueness of the invariant measure. The  main thread between these systems is the existence of a finite number of determining modes (low modes) and a sufficiently rich stochastic forcing to ensure that the low modes are excited. 
This is usually referred as the “effectively elliptic” setting, where all of the presumptively unstable directions
are stochastically forced.
The central idea of the method in \cite{GHMR17} is to introduce a suitable shift in the driving Wiener process to force solutions, which start at different initial conditions, together asymptotically as time goes to infinity. 
For strongly dissipative dynamical systems, in the spirit of \cite{FP}, it is usually enough to control a finite number of unstable directions by introducing a finite-dimensional shift and requiring a sufficiently rich stochastic forcing to ensure that the
unstable modes are excited. 

Starting from these results, 
in \cite{KS} Kulik and Scheutzow exploit  this technique  to prove for the same SPDE's considered in  \cite{GHMR17} an 
 asymptotic stability result too. Moreover the technique of  \cite{GHMR17} for the uniqueness of the invariant measure is improved in \cite{KS}.  In particular Kulik and Scheutzow  still introduce a  control similar to the one considered in \cite{GHMR17} but they have to drop the localization term considered in \cite{GHMR17} (see our Section \ref{uniq_inv_mea_sec} for the details).

These methods  have been successful to prove ergodic properties of the Navier-Stokes equations driven by an \textit{additive} noise. The main aim of our work is to show that those methods are flexible enough to also deal with noises of multiplicative type so to prove uniqueness of the invariant measure and asymptotic stability in an effectively elliptic setting.
To the best of our knowledge, \textit{generalized asymptotic coupling} techniques have so far been used to study the ergodic properties of SPDEs driven by multiplicative-type noises only in the case of delayed equations (see e.g. \cite{HMS11} and \cite{BKS}).
Let us point out that there are works that address ergodic problems for SPDEs with a multiplicative-type noise by different techniques, see e.g. \cite{Oda2008}, \cite{M_mult}, \cite{KS02} and \cite{DonPen}. In particular \cite{Oda2008} and \cite{M_mult} use \textit{coupling} techniques to study the long time behavior of strongly dissipative SPDEs (Navier-Stokes and Ginzburg-Landau) driven by a \textit{bounded} multiplicative noise. In addition to the uniqueness of the invariant measure they also prove the exponential convergence to it in a effectively elliptic setting.
Differently from \cite{Oda2008} and \cite{M_mult} here we deal with more general noises: the covariance operator is  either bounded or satisfies a sublinear or a linear growth condition. 

Therefore in this work we  focus on the stochastic two-dimensional Navier-Stokes equations, but these methods  could be exploited to deal with different types of (strongly dissipative) SPDEs as well.
The Navier-Stokes (NS) equations, considered here with Dirichelt boundary conditions, describe the time evolution of an incompressible fluid and are given by 
\begin{equation}
\label{NS}
\begin{cases}
\partial_t u(t,x)+ \left[-\nu \Delta u(t,x) + (u(t,x)\cdot \nabla)u(t,x)+ \nabla p(t,x)\right]{\rm d}t=f (x)\,{\rm d}t &+G(u(t,x))\partial_t W(t,x),
\\
\nabla \cdot u(t,x)=0, & x \in D, t>0,
\\
u(t,x)=0, & x \in \partial D, t>0,
\\
u(0,x)=u_0(x), & x \in D.
\end{cases}
\end{equation}
Here the unknowns, for any   time $t>0$ and position $x\in D$, are the velocity vector $u(t,x)$  and the scalar pressure $p(t, x)$;
the data are the kinematic viscosity  $\nu>0$ of the fluid, the  initial velocity $u_0$, 
the  deterministic external force $f$  and the  random external force depending on the Wiener process $W$
 and a operator $G$.  We assume $D$ to be an open bounded domain of $\mathbb{R}^2$ with regular boundary.

In the spirit of \cite{GHMR17} and \cite{KS} and as done in \cite{Oda2008} and \cite{KS02} (see also the references in these papers), 
we prove the uniqueness of the invariant measure for system \eqref{NS} and its asymptotic stability under the main requirement on the noise to be non degenerate in the unstable directions, that is, we require  the image of the covariance operator of the noise to contain
a finite number of low modes, corresponding to the 
 the unstable modes.
 A crucial tool is  an estimate in the same spirit as that obtained by    Foias and Prodi 
\cite{FP} for the deterministic Navier-Stokes equations. 
For the stochastic Navier-Stokes equations the Foias-Prodi type estimates  have been proved so far only with an additive or a  bounded multiplicative noise. What we indeed prove is a   Foias-Prodi  estimate  in expected value, showing that 
 a  finite dimensional noise, when chosen in a proper way, allows to synchronize (in the mean)   any two solutions in the limit as $t \rightarrow \infty$. 
 This requires a generalized coupling, obtained by means of a control acting on a finite number of low modes.
 
 It might be useful to revise that in  the literature the Foias-Prodi estimates appear in both forms: pathwise or  
in expectation. When the SPDE has a strong dissipation and an additive noise, then the Foias-Prodi estimate can be proved pathwise. Otherwise, when there is a weak dissipation (e.g.: a damping term $\nu u$ instead of the  Laplacian $-\nu \Delta u$ appearing in \eqref{NS}) or the noise is multiplicative, the Foias-Prodi estimate can be proved in the mean. 
 For these results we refer to   \cite{Odasso}  for  the  nonlinear weakly damped Schr\"odinger equation with additive noise and to  \cite{GHMR21} for  the  weakly damped KdV equation with additive noise; and to 
\cite{Oda2008}  for the Navier-Stokes equations with  a bounded multiplicative noise.

Following the intuition of \cite{Oda2008} we derive the Foias-Prodi estimates in expected value for the Navier-Stoeks equations \eqref{NS} (although formulated in a different form than in \cite{Oda2008} in the case of a bounded noise) and show that they are in fact the crucial ingredient to  readapt  the  \textit{generalized coupling} arguments of \cite{GHMR17} and \cite{KS} to infer uniqueness of the invariant measure and asymptotic stability in presence of multiplicative-type noises. 
As anticipated above, denoting by $P_N(H)$ the subspace spanned by the first $N$ modes and assuming that $N$ is large enough to contain the 
unstable modes, this technique  requires that the range of $G(u)$ contains $P_N(H)$. 
In addition, when the operator $G$ in front of the Wiener process has a linear growth, it will also be necessary to impose that the  viscosity coefficient $\nu$ somehow balances the intensity of the multiplicative part of the noise. In this case, the existence of the invariant measure, its uniqueness and its asymptotic stability require gradually to strengthen this condition on $\nu$. 
We obtain different type of Foias-Frodi estimates depending on the assumptions we make on the noise. In the case of a bounded noise we get an exponential decay while in the case of sublinear or linear growth noise we get a polynomial decay. The substantial difference in the latter two cases is that in the case of a sublinear growth noise the time decay goes as $t^{-p}$, when $t\to+\infty$, for  an arbitrary power $p>0$; in the case of a linear growth noise the range of admissible  parameters $p$ is related to the viscosity coefficient and to the intensity of the multiplicative part of the noise. 
 
The Foias-Prodi estimates we obtain are an interesting result in themselves, 
and we hope to use them also to obtain quantitative mixing results. This problem is currently under investigation.

We conclude by briefly summarizing the content of the paper. In Section \ref{setting_sec} we introduce the mathematical setting, state the assumptions and the main results. In Section \ref{prel_res_sec} we recall some results concerning well-posedness of the stochastic Navier-Stokes equations  \eqref{NS}. 
In Section \ref{Foias-Prodi_section} we derive the Foias-Prodi type estimates in expected value. Section \ref{uniq_inv_mea_sec} provides the proof of the existence, uniqueness and asymptotic stability of the invariant measure. Some remarks are collected in Section \ref{final_rem_sec}. 
In Appendix \ref{app_A} are collected some apriori estimates and in Appendix \ref{dim_lem} the proof of a technical lemma  is provided.

%\textcolor{blue}Fare dei commenti/confronti sui seguenti punti:\begin{itemize}\item Dong-Peng 2018 dicono che estendono al caso moltiplicativo il risultato di Hairer-Mattingly. Inoltre provano il risultato di Debussche ma con tecniche diverse. Citano anche \item Odasso 2008 :\begin{itemize}\item dice che lavora sotto ipotesi meno restrittive dei precedenti articoli, e.g. non usa l'ipotesi che il noise sia diagonal, come Mattingly.\item intuitivamente, dovremmo otterenere un risultato simile, sotto le stesse ipotesi del noise, ma usando un asymptotic coupling invece che coupling. Vedi intro di Dong-Pensg 2018:  citando Odasso, dicono che lui lavora sotto H1-H3. H1 richiede che il noise sia BOUNDED!!!, H3 dovrebbe servire per usare Girsanov (vedi intro Dong-Peng 2018 per un riassunto).\item Lavora anche Ginzburg-Landau, noi riusciamo?\end{itemize}\item Mattingly 2002: anche lui lavora con un noise bounded (mi sembra): va confrontata con la lui la convergenza alla misura invariante: a occhio, a lui \'e esponenziale a noi polinomiale...per questo probema vedi i Remark in Odasso: con il noise bounded si divrebbe avere descrescita esponenziale, con il noise sublineare, decrescita polinomiale.\item Kuksin-Shirikyan 2002: stesso discorso fatto per Mattingly. In entrambi i casi mi sembra essenziale il fatto di avere un noise bounded per ottenere un decadimento esponenziale.\end{itemize}}

\section{Setting and main results}
\label{setting_sec}
 In this Section we fix the notations, explain the assumptions, formulate the framework of our problem and state the main results.
 
 In the sequel, given two Banach spaces $E$ and $F$, we denote by $L(E,F)$ the space of all linear bounded
operators $B: E\to F$ and abbreviate $L(E):=L(E,E).$  If $H$ and $K$ are separable Hilbert spaces, we employ the symbol
$L_{HS}(H, K)$ for the space of Hilbert-Schmidt operators from $H$ to $K$. If $(A, \mathcal{A}, \mu)$ is a finite measure space, we denote by $L^p(A,E)$ 
the space of $p$-Bochner integrable functions, for any $p \in [1, \infty)$. 
Given the Hilbert space $H$, for a fixed $T>0$, by $C([0,T];H)$ we denote the space of strongly continuous functions from $[0,T]$ to $H$ whereas $C_w([0,T];H)$ stands for the space of all continuous functions from the interval $[0,T]$ to  the space $H$ endowed with the weak topology.
\newline
If functions  $a,b\ge 0$ satisfy the inequality $a \le C(A) b$ with a constant $C(A)>0$ depending on the expression $A$, we write $a \lesssim_A b$; for a generic constant we put no subscript.
Everywhere $C$ denotes a generic constant; if needed, we specify the parameters on which it depends.

We consider the usual abstract form of equations \eqref{NS} (see, e.g., \cite{Temam2001} for further details). 
Let $\mathcal{V}$ be the space of smooth  and divergence-free vector fields $u:D\to\mathbb R^2$  with compact support strictly contained in $D$. 
We denote by $H$ and $V$ the closure of $\mathcal{V}$ in $[L^2(D)]^2$ and in $[H^1(D)]^2$, respectively. We denote by $\|\cdot\|_H$ and $\langle \cdot, \cdot \rangle$ the norm and the inner product in $H$. By $V^*$ we denote the dual space of $V$ and by $\langle \cdot, \cdot \rangle$ we denote the dual pairing between $V$ and $V^*$ when no confusion may arise. We set $\mathcal{D}(A):=[H^2(D)]^2 \cap \mathcal{V}$, and define the linear operator $A: \mathcal{D}(A) \subset H \rightarrow H$ as $Au=-\Pi\Delta u$, where $\Pi$ is the projection from $[L^2(D)]^2$ to $H$. Since $V$ coincides with $\mathcal{D}(A^\frac 12)$, we endow $V$ with the norm $\|u\|_V=\|A^{\frac 12}u\|_H$. 
The operator $A$ is a positive selfadjoint operator in $H$ with compact resolvent; we denote by 
$\{\lambda_j\}_{j\in \mathbb N}$ the eigenvalues of $A$ and by  $\{e_j\}_{j\in \mathbb N}$ the corresponding eigenvectors of $A$ that form a complete orthonormal system in $H$.  Moreover
$0<\lambda_1 \le \lambda_2 \le \cdots$ and
\[
\lim_{j\to\infty}\lambda_j=+\infty.
\]
 We recall  the Poincar\'e inequality
\begin{equation}
\label{lambda_1}
\|u\|^2_V \ge \lambda_1 \|u\|^2_H.
\end{equation}
Denoting by $P_N$ and $Q_N$ the orthogonal projection in $H$ onto the space $Span\{e_n\}_{1\le n\le N}$ and onto its complementary, respectively, we have the generalized Poincar\'e inequalities 
\begin{equation}
\label{gen_Poin}
\|P_Nu\|_V^2 \le \lambda_N\|P_N u\|^2_H, \qquad \|Q_Nu\|^2_H \le \frac{1}{\lambda_N}\|Q_Nu\|^2_V
\end{equation}
that hold for all sufficiently smooth $u$ and any $N\ge 1$.
\\
We define the bilinear operator $B: V \times V \rightarrow V^*$ as 
\begin{equation*}
\langle B(u,v), z\rangle:=\int_D(u(x)\cdot \nabla )v(x) \cdot z(x)\, {\rm d} x.
\end{equation*}
It holds
\begin{equation}
\label{B=0}
\langle B(u,v), v\rangle=0,
\end{equation}
\begin{equation*}
\langle B(u,v), z\rangle= -\langle B(u,z), v\rangle.
\end{equation*}

As far as the random forcing term is concerned, we always consider a
 filtered probability space $(\Omega,\mathcal{F},\mathbb{P},\mathbb{F})$, where the filtration $\mathbb{F}=\bigl(\mathscr{F}_t\bigr)_{t\geq 0}$ is right continuous and 
 $\mathscr{F}_0$ contains all $\mathbb{P}$-null events.
 Moreover 
 $W$ is an $U$-cylindrical $\mathbb{F}$-Wiener process, where $U$  is a separable real Hilbert space with  an orthonormal basis  $(f_n)_{n\in\mathbb{N}}$ (see, e.g., details in \cite{DapZab}).

Moreover, we will work under the following assumptions on the operator $G$ characterizing the noise. 
\begin{Assumption}
\label{assumption-stochastic}
 $\;$\\[-3mm]
\begin{itemize}
\item [\textbf{(G1)}]  \label{G1} 
$G:{H} \to L_{HS}(U,H)$ is a Lipschitz continuous operator,  i.e.
 \begin{equation} \label{Lipschitz_G}
  \exists \ L_G >0 : \quad  \|G(u_1)-G(u_2)\|_{L_{HS}(U,H)} \le L_G\|u_1-u_2\|_H \qquad \forall u_1, u_2 \in H.
 \end{equation}
\item[\textbf{(G2)(i)}] \label{item:1} 
 There exists a non negative constant $K_1$ such that
 \begin{equation*}\label{bounded}
  \|G(u)\|_{L_{HS}(U,H)}\le K_1,\quad \forall u \in H.
 \end{equation*}
\item[\textbf{(G2)(ii)}] \label{item:2} 
  There exist non negative constants $K_2, \tilde{K_2}$ and $\gamma \in (0,1)$ such that
 \begin{equation*}
  \|G(u)\|_{L_{HS}(U,H)}\le K_2+ \tilde{K}_2\|u\|_H^{\gamma}, \quad \forall u \in H.
 \end{equation*}
\item[\textbf{(G2)(iii)}] 
\label{item:3} 
  There exist non negative constants $K_3, \tilde{K_3}$ such that
 \begin{equation*}
  \|G(u)\|_{L_{HS}(U,H)}\le K_3+ \tilde{K}_3\|u\|_H, \quad \forall u \in H.
 \end{equation*}
\item[\textbf{(G3)}]  \label{G3}
There exists a 
measurable map $g:H \rightarrow L(H,U)$  such that
\begin{equation}
\label{bound_g}
\sup_{u \in H}\|g(u)\|_{L(H,U)} < \infty
\end{equation}
and
\begin{equation}
\label{GgM}
G(u)g(u)=P_M \qquad \forall u \in H
\end{equation}
for a positive integer $M$.
\end{itemize}
\end{Assumption}

In the sequel, when we say that assumption (G2) holds we mean that 
one of the three assumptions (G2)(i), (G2)(ii), (G2)(iii) holds.

\begin{notation}
Throughout the paper we will reserve the symbol $M$ to denote the integer that appears in Assumption (G3). To infer the uniqueness of the invariant measure and the qualitative mixing result we will require $M$ to be sufficiently large.
\end{notation}

\begin{remark}
The existence of a map $g: H \rightarrow L(H,U)$  fulfilling \eqref{GgM} is equivalent to the following property 
\begin{equation*}
P_MH \subseteq \text{{Im}}\ G(u) \qquad \forall u \in H.
\end{equation*}

Thus  assumption (G3) can be seen as a non degeneracy condition on the low modes. We refer to \cite[Remarks 3.1 and 3.2]{Oda2008} for more details. 
\end{remark}

\begin{example}[for assumption (G3)]
\label{cond_phi_k_rem}
Let us recall that by $\{f_k\}_{k \in \mathbb{N}}$ and $\{e_k\}_{k \in \mathbb{N}}$ we denote orthonormal basis in $U$ and $H$, respectively. 
Suppose that for any $k\in \mathbb N$  there exists a mapping $\phi_k:H \rightarrow \mathbb{R}$
such that for some $M \in \mathbb{N}$,
\begin{equation}
\label{G(u)}
G(x)f_k=\phi_k(x)e_k,\qquad  \ \forall \ k\le M, \ \forall\ x \in H,
\end{equation}
and
\begin{equation*}
0 \ne \phi_k(x), \qquad \forall \ k\le M, \quad \forall\ x \in H. %dovrebbe garantirmi l'invertibilt\'a.
\end{equation*}
Take 
%\begin{footnote}{\textcolor{blue}{La condizione che sia zero sugli high modes mi garantisce che $g(x) \in L(H,H)$}}\end{footnote}
\begin{equation}
\label{g_es}
g(x)e_k=
\begin{cases}
\phi_k(x)^{-1}f_k & \text{if} \ k\le M,
\\
0 & \text{otherwise},
\end{cases}
\end{equation}  
with 
\begin{equation}\label{somma-finita_phi}
\sup_{x \in H} \sum_{k\le M} |\phi_k(x)^{-1}|< \infty.
\end{equation}
Then $g$ satisfies \eqref{bound_g} and \eqref{GgM} in Assumption (G3). 
In fact, 
\[\begin{split}
\sup_{u \in H}\|g(u)\|_{L(H,U)}
&=\sup_{u \in H}\sup_{h \in H,\|h\|_H \le 1}\|g(u)h\|_{U}
\\
&=\sup_{u \in H}\sup_{h \in H,\|h\|_H \le 1}\left\Vert\sum_{k \in \mathbb{N}}\langle h, e_k\rangle g(u)e_k\right\Vert_{U}
\\&
=\sup_{u \in H}\sup_{h \in H,\|h\|_H \le 1}\left \Vert\sum_{k \le M}\langle h, e_k\rangle \phi_k(u)^{-1}f_k\right\Vert_{U}
\\
&\le \sup_{u \in H}\sup_{h \in H,\|h\|_H \le 1}\sum_{k \le M}\|h\|_H \|e_k\|_H\|f_k\|_U| \phi_k(u)^{-1}|
\\&
\le \sup_{u \in H}\sum_{k \le M}| \phi_k(u)^{-1}|
\end{split}\]
and the latter quantity is finite by \eqref{somma-finita_phi}.
Moreover, let $v \in H$; then, for any $u \in H$
\begin{multline*}
G(u)g(u)v= G(u)g(u)\left(\sum_{k \in \mathbb{N}}\langle v, e_k\rangle e_k \right)
=G(u)\left(\sum_{k \in \mathbb{N}}\langle v, e_k\rangle g(u)e_k \right)
=
G(u)\left(\sum_{k \le M }\langle v, e_k\rangle \phi_k(u)^{-1}f_k \right)
\\
=\sum_{k \le M }\langle v, e_k\rangle \phi_k(u)^{-1}G(u)f_k 
=\sum_{k \le M }\langle v, e_k\rangle \phi_k(u)^{-1}\phi_k(u)e_k =P_Mv.
\end{multline*}
\end{example}
 
We provide  now concrete examples for operators $G$ satisfying Assumption \ref{assumption-stochastic}.
\begin{example}[for Assumption \ref{assumption-stochastic}]
We take $U=H$. 
\begin{itemize}
\item Consider the mappings
\begin{equation*}
\phi_k(x):= \frac{\sqrt{\|x\|_H^2+1}}{k+1}, \quad k \in \mathbb{N}
\end{equation*}
and define the operator $G$ as in \eqref{G(u)}. Then $G$ satisfies 
(G1) and  (G2)(iii). Moreover, \eqref{bound_g} and \eqref{GgM} in (G3) are satisfied for any finite $M$  by choosing $g$ as in \eqref{g_es}. 
%\begin{multline*}\sup_{u \in H}\|g(u)\|_{L(H,H)}\le \sup_{u \in H}\sum_{k \le M}| \phi_k(u)^{-1}|=\sup_{u \in H}\sum_{k \le M}\frac{k+1}{\sqrt{\|u\|^2_H+1}}\lesssim_M \sup_{u \in H}\frac{1}{\sqrt{\|u\|^2_H+1}}< \infty.\end{multline*}
%
\item Let $\gamma\in (0,1)$. Consider the mappings 
\begin{equation*}
\phi_k(x):= \frac{\sqrt{\left(\|x\|_H^2+1\right)\pmb{1}_{(\|x\|_H \le 1)}+\left( \|x\|_H^{2\gamma}+1\right)\pmb{1}_{(\|x\|_H > 1)}}}{k+1}, \quad k \in \mathbb{N}
\end{equation*}
and define the operator $G$ as in \eqref{G(u)}. Then $G$ satisfies 
(G1) and (G2)(ii). Moreover, \eqref{bound_g} and \eqref{GgM} in (G3) are satisfied for any finite $M$  by choosing $g$ as in \eqref{g_es}. 
\item Consider the mappings
\begin{equation*}
\phi_k(x):= \frac{\sqrt{\left(\|x\|_H^2+1\right)\pmb{1}_{(\|x\|_H \le 1)}+\pmb{1}_{(\|x\|_H > 1)}}}{k+1}, \quad k \in \mathbb{N}
\end{equation*}
and define the operator $G$ as in \eqref{G(u)}. Then $G$ satisfies 
(G1) and  (G2)(i). Moreover, \eqref{bound_g} and \eqref{GgM} in (G3) are satisfied  for any finite $M$  by choosing $g$ as in \eqref{g_es}. 
\end{itemize}
\end{example}

We can rewrite problem \eqref{NS} in the  abstract form 
\begin{equation}
\label{NS_abs}
\begin{cases}
{\rm d}u(t) + \left[\nu Au(t)+B(u(t),u(t))\right]\,{\rm d}t=f \, {\rm d}t+G(u(t))\,{\rm d}W(t),\qquad\qquad t>0
\\
u(0)=u_0
\end{cases}
\end{equation}
We assume $\nu>0$, $u_0\in H$  and $ f \in V^*$ independent of time. 

Here is our main result on the stochastic Navier-Stokes equation \eqref{NS_abs}; for a more precise statement see Proposition \ref{prop_mis_inv}, and Theorems \ref{unique_thm} and \ref{asy_sta}.

\begin{theorem}
\label{thm_intro}
Assume (G1).
\begin{itemize}
\item [a)]
If  (G2)(i) or  (G2)(ii) hold, then there exists  a positive integer  $\bar N$, 
 depending on $\nu, f$ and $G$,  such that,
 whenever  (G3) holds for some $M \ge \bar N$, there exists a unique invariant measure which is asymptotically stable.
\item [b)] 
If (G2)(iii) holds and $\nu>\frac{\tilde K_3^2}{2\lambda_1}$,
 then there exists at least one invariant measure. Moreover, there exists a positive integer  $\bar N$, depending on $\nu, f$ and $G$,
 such that,  if   (G3) holds for some $M \ge \bar N$, then the invariant measure is unique provided $\nu>\frac{3\tilde K_3^2}{2\lambda_1}$ and it is asymptotically stable provided $\nu>\frac{11\tilde K_3^2}{2\lambda_1}$.
\end{itemize}
\end{theorem}

\begin{remark}
\label{rem_cases}
%\begin{itemize}\item [i)]It is well known that the Lipschitz assumption \eqref{Lipschitz_G} implies (G2)(iii). We explicitly write it down in the assumptions for the sake of exposition.
%\item [ii)]
Notice that
\[
 (G2)(i)\Longrightarrow(G2)(ii)\Longrightarrow (G2)(iii).
\]
Indeed, (G2)(i) is a particular case of (G2)(ii): take $K_2=K_1$ and $\tilde K_2=0$.

 Furthermore, by  the Young inequality we have 
 \[
 \|u\|_H^{\gamma}\le \varepsilon \|u\|_H+(1-\gamma)\left(\frac \gamma \varepsilon\right)^{\frac \gamma{1-\gamma}}
 \]
 for any positive $\varepsilon$. This shows the other implication.

We deduce that 
 if we are able to prove a result working under assumption  (G2)(iii) then the same result will hold also under (G2)(i) and (G2)(ii).

The statement of Theorem \ref{thm_intro} explains why we  state Assumption (G2) separating the three cases. 
Under the stronger assumption (G2)(i) or (G2)(ii) we prove the existence of a unique invariant measure, which is asymptotically stable, without any requirement on the viscosity $\nu$.
Notice that the case (G2)(i) corresponds to the case studied by Odasso in \cite{Oda2008}
%(see in particular his Assumption (H0))
 where the author, with different techniques, obtains the same results as we do (in fact, he also proves exponential mixing). 
 Things are more delicate under  the weaker assumption (G3)(iii): the existence of the invariant measure, its uniqueness and its asymptotic stability require gradually narrower assumptions on the viscosity  $\nu$. 
 A strong enough dissipation is required to balance the intensity of the multiplicative part of the noise $\tilde K_3$ more and more consistently.
 
Another reason to consider  three different hypotheses  concerns the Foias-Prodi estimates that we will derive in Theorem \ref{FP_cor}: depending on the type of assumption (G2) on the noise, we will get different decays in time (exponential or polynomial). We believe that these different decays will eventually lead to different types of quantitative mixing; this is under investigation at the moment.
\end{remark}

%We will provide separate proofs for the uniqueness of the invariant measure and the mixing property. One could observe that the uniqueness of the invariant measure follows as a direct consequence of the mixing property; we decided to provide a separate proof for the uniqueness of the invariant measure to highlight how the technique developed in \cite{GHMR17} can be readapted to deal also with multiplicative-type noises. 

\section{Well posedness results}
\label{prel_res_sec}

In this Section we collect the results concerning the well posedness of system \eqref{NS_abs} under the Assumptions  (G1) and (G2). These are classical results.
Keeping in mind Remark \ref{rem_cases} it is enough to prove them under Assumptions  (G1) and (G2)(iii).

First, the solutions can be   weak or  strong solutions, in the probabilistic sense.

\begin{definition}
\label{mart_sol_def}
We say that there exists a martingale solution of the Navier-Stokes equation
\eqref{NS_abs} on the interval $[0,T]$ and with initial velocity $u_0\in H$ if there exist
a stochastic basis $(\widetilde \Omega, \widetilde{\mathcal{F}}, \widetilde{\mathbb{P}},\widetilde{\mathbb{F}})$, a $U$-cylindrical Wiener process $\widetilde W$,
and a progressively measurable process $u:[0,T]\times \widetilde \Omega \rightarrow H$ with $\widetilde{\mathbb{P}}$ a.e. paths 
\begin{equation*}
v \in C([0,T];H) \cap L^2(0,T;V)
\end{equation*}
such that $\widetilde{\mathbb{P}}$-a.s., the identity 
\begin{multline*}
\langle u(t), \psi\rangle-\int_0^t \langle A^{\frac 12}u(s), A^{\frac 12}\psi\rangle\, {\rm d}s + \int_0^t \langle B(u(s),u(s)), \psi\rangle\, {\rm d}s 
=\langle u_0, \psi\rangle + \langle f, \psi\rangle t + \langle \int_0^t G(u(s))\, {\rm d}\widetilde{W}(s), \psi\rangle
\end{multline*}
holds true for any $t \in [0,T]$, $\psi \in V$.
\end{definition}

\begin{definition}
\label{mart_sol_def}
Given a stochastic basis $(\Omega, \mathcal{F}, \mathbb{P}, \mathbb{F})$ and  a $U$-cylindrical Wiener process $W$, a strong solution of  the Navier-Stokes equation \eqref{NS_abs} on the interval $[0,T]$ 
with initial velocity $u_0\in H$ is an $H$-valued contiunuous $\mathbb{F}$-adapted process $u$ with $\mathbb{P}$-a.e. path in $L^2(0,T;V)$
such that ${\mathbb{P}}$-a.s., the identity 
\begin{multline*}
\langle u(t), \psi\rangle-\int_0^t \langle A^{\frac 12}u(s), A^{\frac 12}\psi\rangle\, {\rm d}s + \int_0^t \langle B(u(s),u(s)), \psi\rangle\, {\rm d}s 
=\langle u_0, \psi\rangle + \langle f, \psi\rangle t + \langle \int_0^t G(u(s))\, {\rm d}W(s), \psi\rangle
\end{multline*}
holds true for any $t \in [0,T]$, $\psi \in V$.
\end{definition}

Now  we consider the existence of martingale solutions.

%PROP
\begin{proposition}
\label{prop_exis_mart_sol}
Under Assumptions (G1) and (G2), for any $T>0$ there exists a martingale solution to problem \eqref{NS_abs} which satisfies, for any $q \ge 2$,
\begin{equation}
\label{q_mom_sol_0}
\widetilde{\mathbb{E}}\left[ \|u\|^q_{L^{\infty}(0,T;H)}\right] < \infty.
\end{equation}
\end{proposition}
%PROOF
\begin{proof}
Assuming  (G1) and (G2)(iii), in \cite[Thorem 3.1]{Fla_Gat} the existence of a martingale solution  is proved in any space dimension $d \ge 2$, with $\widetilde{\mathbb{P}}$ a.e. paths 
$v \in C_w([0,T];H) \cap L^2(0,T;V)$. Arguing as in \cite[Lemma 7.2]{BrzMot}, in dimension $d=2$, one can prove the additional regularity $u \in C([0,T];H)$ $\widetilde{\mathbb{P}}$-a.s.
Estimate \eqref{q_mom_sol_0} is proved in \cite[Appendix A]{Fla_Gat}. %\textcolor{red}{Check: Qua si dimostra per Galerkin: vale anche per il limite? Controllare meglio. Su questo mi \'e rimasto un dubbio.}

Keeping in mind Remark \ref{rem_cases}
we get that the result is true when we assume any of the three (G2) conditions.
\end{proof}

Then we consider the pathwise uniqueness.
\begin{proposition}
\label{prop_uniq_sol}
Let $T>0$. Let Assumptions (G1) and (G2) hold. Let $(\widetilde\Omega, \widetilde{\mathcal{F}}, \widetilde{\mathbb{P}},  \widetilde{\mathbb{F}},u_i)$, $i=1,2$ be two martingale solutions to \eqref{NS_abs} with the same initial velocity. Then $\widetilde{\mathbb{P}}(u_1(t)=u_2(t)\, \ \text{for all} \ t \in [0,T] )=1$, that is solutions to equation \eqref{NS_abs} are pathwise unique.
\end{proposition}

The proof of the result is based on the following technical lemma whose proof is postponed to Appendix \ref{dim_lem}.
%LEMMA
\begin{lemma}
\label{techn_lemm}
Let Assumptions (G1)-(G2)(iii) hold. Let $(\widetilde\Omega, \widetilde{\mathcal{F}}, \widetilde{\mathbb{P}},  \widetilde{\mathbb{F}},u_i)$, $i=1,2$ be two martingale solutions to \eqref{NS_abs} with initial velocities $x, y \in H$, respectively. 
Then
\begin{equation*}
\widetilde{\mathbb{E}}\left[ e^{-\left(L_G^2t- \lambda_1 \nu t+\frac{1}{\nu}\int_0^t\|u_1(s)\|^2_V\, {\rm d}s \right)}\|u_1(t)-u_2(t)\|^2_H\right] \le \|x-y\|_H^2.
\end{equation*}
\end{lemma}
%PROOF
\begin{proof} $[$of Proposition \ref{prop_uniq_sol}$]$
Keeping in mind Remark \ref{rem_cases}
we proceed assuming (G2)(iii).
Lemma \ref{techn_lemm} yields
\begin{equation*}
\widetilde{\mathbb{E}}\left[ e^{-\left(L_G^2t- \lambda_1 \nu t+\frac{1}{\nu}\int_0^t\|u_1(s)\|^2_V\, {\rm d}s \right)}\|u_1(t)-u_2(t)\|^2_H\right] \le 0.
\end{equation*}
So 
\begin{equation*}
e^{-\left(L_G^2t- \lambda_1 \nu t+\frac{1}{\nu}\int_0^t\|u_1(s)\|^2_V\, {\rm d}s \right)}\|u_1(t)-u_2(t)\|^2_H= 0, \quad \widetilde{\mathbb{P}}-a.s..
\end{equation*}
Thus, if we take a sequence $\{t_k\}_{k=1}^\infty$ which is dense in $[0,T]$ we have 
\begin{equation*}
\widetilde{\mathbb{P}}\left( \|u_1(t_k)-u_2(t_k)\|_H=0, \ \forall \ k \in \mathbb{N}\right)=1.
\end{equation*}
Since a.e.  path of the  solution process belongs to $C([0,T],H)$ we infer $\widetilde{\mathbb{P}}\left( \|u_1(t)-u_2(t)\|_H=0, \ \forall \ t \in[0,T] \right)=1$ and this concludes the proof.
\end{proof}

Keeping in mind Propositions \ref{prop_exis_mart_sol} and \ref{prop_uniq_sol} and \cite{Uniqueness}, which ensures that existence of a martingale solution and pathwise uniqueness yield existence of a unique strong solution, we get

%THEOREM
\begin{theorem}
\label{ex_uniq_sol}
Under Assumptions  (G1) and (G2) there exists a unique strong solution to problem \eqref{NS_abs} with $\mathbb{P}$-a.e. paths in $C([0, + \infty);H)\cap L^2_{loc}(0, \infty;V)$ that satisfies, for any $T>0$ and $q \ge 2$,
\begin{equation}
\label{q_mom_sol}
\mathbb{E}\left[ \|u\|^q_{L^{\infty}(0,T;H)}\right] < \infty.
\end{equation}
\end{theorem}

\section{Foias-Prodi estimates in expectation}
\label{Foias-Prodi_section}

This Section is devoted to establishing a Foias-Prodi type estimate for the Navier-Stokes equation \eqref{NS_abs} that holds in expectation. This result will serve as a crucial technical tool for the arguments establishing the uniqueness of the invariant measure and the qualitative mixing result. 

The Foias-Prodi estimates describe the following property for an infinite dimensional dynamical system:
 given any two solutions, if they 
  synchronize in the limit as $t \rightarrow + \infty$ on a sufficient (but finite) number of components, i.e. the low modes,
 then in fact all components synchronize. In other words, the dynamics of the high modes is asymptotically enslaved to the dynamics of the low modes.
% \sout{This type of estimates was derived for the first time by Foias and Prodi in \cite{FP} for the deterministic 2D Navier-Stokes equations. In the stochastic setting estimates of this type provide a fundamental technical tool to infer ergodic property of the solution. Classically, these type of estimates appear in a pathwise formulation for strongly dissipative SPDEs (e.g. Navier-Stokes equations) driven by an \textit{additive} noise. However, for weakly dissipative SPDEs (like damped nonlinear Schr\"odinger equation, see \cite{Odasso}, or KdV equation, see \cite{GHMR21}, where the dissipation term is not given by a Laplace operator but by an operator of  zero  order) driven by an \textit{additive} noise, the Foias-Prodi estimate cannot be proved pathwise but only in mean value. A similar issue appears now: even if the SPDE has a strong dissipation, the Foias-Prodi estimate cannot be proved pathwise since the noise is of multiplicative type. }
What we get is that any two solutions, with different initial velocities, converge to each other as $t\to+\infty$ if a control acts on a sufficient  finite number of components; the convergence is in mean value.

We proceed as follows.
Given $G$ satisfying assumptions (G1) and (G2) and $u_0 \in H$, let $u=u(u_0)$ denote the  solution of  the Navier-Stokes equation \eqref{NS_abs}.
Given $\lambda>0, N>0$ and $v_0\in H$, let $v=v(v_0,u_0)$ denote the corresponding solution of 
\begin{equation}
\label{NS_abs_nud}
\begin{cases}
{\rm d}v(t) + \left[\nu Av(t)+B(v(t),v(t))\right]\,{\rm d}t
= 
f \,{\rm d}t +G(v(t))\,{\rm d}W(t)+ \lambda P_N(u(t)-v(t))\,{\rm d}t, \qquad t>0
\\
v(0)=v_0
\end{cases}
\end{equation}
where $P_N$ is the orthogonal projection from  $H$ onto the space $Span\{e_n\}_{1\le n\le N}$. Here $\lambda>0$ is a parameter to be  suitably chosen later on. 
\begin{notation}
\label{notation_N}
Throughout the paper we will reserve the symbol $N$ to indicate the dimension of the projected space $P_NH$ where the control $\lambda P_N(u-v)$  acts. 
\end{notation}
We will refer to \eqref{NS_abs_nud} as the \textit{nudged equation} corresponding to the  Navier-Stokes equation \eqref{NS_abs}. The well posedness of   \eqref{NS_abs_nud} can be trivially proved for \eqref{NS_abs_nud}: the additional term $\lambda P_N(u-v)=\lambda P_N u-\lambda P_Nv$ 
does not crucially impact the well-posedness estimates (see, e.g., \cite[Remark 8]{KS}).

The effect of the \textit{nudging term} $ \lambda P_N(u-v)$ is to drive $v$ towards $u$ on $P_NH$ that is on the low modes; the Foias-Prodi estimates (in expectation) will in fact quantity how many modes need to be activated in order to synchronize the full solution. More in details, we will show in Theorem \ref{FP_cor} that, provided $N$ is taken sufficiently large, $\mathbb{E}\left[ \|u(t)-v(t)\|^2_H\right]$ decays in time, as $t\to+\infty$, with different rates according to the different assumptions  on the covariance operator of the noise. The idea of the proof, inspired by \cite{GHMR21}, is as follows: we show (see Subsection \ref{FP_sec}) that for certain stopping time $\tau_{R, \beta}$ that controls the growth of the solution to \eqref{NS_abs}, the expectation $\mathbb{E}\left[\pmb{1}_{(\tau_{R, \beta=\infty)}} \|u(t)-v(t)\|^2_H\right]$ decays with an exponential rate in time (see Corollary \ref{cor_tau_R_beta}). Then we exploit energy estimates for the solutions to \eqref{NS_abs} to prove that the probability such stopping time $\tau_{R, \beta}$ remains finite decays in the cut-off parameter $R$ (see Subsection \ref{decay_sec}). According to the different assumptions on the covariance of the noise (G2)(i), (ii) or (iii) we obtain different types of decay in $R$, see Propositions \ref{tau_exp}, \ref{tau_pol_ii} and \ref{tau_pol_iii}. We highlight that in order to obtain a decay in the parameter $R$ we will need to take $N$  sufficiently large. The Foias-Prodi estimates (see Theorem \ref{FP_cor}) easily  follow, combining the results of Corollary \ref{cor_tau_R_beta} and Propositions \ref{tau_exp}, \ref{tau_pol_ii} and \ref{tau_pol_iii}.

\subsection{A preliminary estimate in expected value}
\label{FP_sec}

Let us start with the following preliminary result.
\begin{proposition}
\label{FP_thm}
Assume  (G1) and (G2).  
If we take $\lambda= \frac{\nu \lambda_N}{2}$ in the nudged equation \eqref{NS_abs_nud}, then for any $u_0, v_0 \in H$, the estimate 
\begin{multline}
\label{FP_est}
\mathbb{E}\left[\exp\left(\left(\frac{\nu\lambda_N}{2}-L_G^2\right) (t\wedge \tau)-\frac{1}{\nu}\int_0^{t\wedge \tau}\|u(s)\|^2_V\, {\rm d}s \right) \|u(t \wedge \tau)-v(t \wedge \tau)\|^2_H 
\right.
\\
 \left.+ \frac{\nu\lambda_N}{2}\int_0^{t \wedge \tau}\exp\left(\left(\frac{\nu\lambda_N}{2}-L_G^2\right)s-\frac{1}{\nu}\int_0^s\|u(\zeta)\|^2_V\, {\rm d}\zeta\right)\|u(s)-v(s)\|^2_H \, {\rm d}s\right] \le \|u_0-v_0\|^2_H
\end{multline}
holds for any stopping time $\tau\ge 0$ and any $t \ge 0$. Here $u=u(u_0)$ and $v=v(v_0,u_0)$ obey equations \eqref{NS_abs} and \eqref{NS_abs_nud} respectively.
\end{proposition}
\begin{proof}

Given $u, v$ satisfying equations \eqref{NS_abs} and \eqref{NS_abs_nud} respectively, we obtain the evolution of the difference  $r:=u-v$
\begin{equation}
%\label{r_eq}
\begin{cases}
{\rm d}r+ \left[ \nu Ar+B(r,u)+B(v,r)+\lambda P_Nr\right]\,{\rm d}t=(G(u)-G(v))\,{\rm d}W
\\
r(0)=u_0-v_0.
\end{cases}
\end{equation}
We apply the It\^o formula to the functional $\|r(t)\|^2_H$. Exploiting \eqref{B=0}, we obtain, for any $t\ge 0$, $\mathbb{P}$-a.s.,
\begin{align*}
\frac 12 {\rm d}\|r(t)\|^2_H+\nu \|\nabla r(t)\|^2_H \, {\rm d}t
&=\left[- \langle B(r(t), u(t)), r(t)\rangle- \lambda \|P_Nr\|^2_H+ \frac12\|G(u(t))-G(v(t))\|^2_{L_{HS}(U,H)}\right]\, {\rm d}t\\
& \qquad + \langle r(t), [ G(u(t))-G(v(t))]\, {\rm d}W(t)\rangle.
\end{align*}
The Gagliardo-Nierenberg and the Young inequality yield
\begin{align*}
|\langle B(r(t), u(t)), r(t)\rangle| 
&\le \|\nabla u(t)\|_H \|r(t)\|^2_{L^4} \le \|\nabla u(t)\|_H \|r(t)\|_H\|\nabla r(t)\|_H 
\\
& \le \frac{\nu}{2}\|\nabla r(t)\|^2_H + \frac{1}{2\nu}\|u(t)\|^2_V\|r(t)\|^2_H.
\end{align*}
Therefore, from \eqref{Lipschitz_G} we infer 
\begin{equation*}
\frac12{\rm d}\|r(t)\|^2_H + \left[\frac{\nu}{2}\|\nabla r(t)\|^2_H+ \lambda \|P_Nr(t)\|^2_H\right]\, {\rm d}t\le \left( \frac{L_G^2}{2}+ \frac{1}{2\nu}\|u(t)\|^2_V\right)\|r(t)\|^2_H\, {\rm d}t + {\rm d}M(t),\end{equation*}
where we set
\begin{equation*}
M(t):=\int_0^t\langle  r(s) ,\, [G(u(t))-G(v(s))]{\rm d}W(s)\rangle.
\end{equation*}
Thanks to the generalized inverse Poincar\'e inequality \eqref{gen_Poin} we obtain 
\begin{align*}
\frac{\nu}{2}\|\nabla r(t)\|^2_H +\lambda \|P_Nr(t)\|^2_H 
&\ge \frac{\nu}{2}\|\nabla Q_Nr(t)\|^2_H +\lambda \|P_Nr(t)\|^2_H 
\\
&\ge \frac{\nu\lambda_N}{2}\|Q_Nr(t)\|^2_H +\lambda \|P_Nr(t)\|^2_H 
\end{align*}
and  by choosing $\lambda= \frac{\nu}{2}\lambda_N$ 
the latter sum equals $\frac{\nu\lambda_N}{2} \|r(t)\|^2_H$.

Thus we finally obtain
\begin{equation}
\label{sti_0}
{\rm d}\|r(t)\|^2_H + \left(\nu\lambda_N- L_G^2-\frac{1}{\nu}\|u(t)\|^2_V\right) \|r(t)\|^2_H\,{\rm d}t \le {\rm d}M(t).
\end{equation}
We set
\begin{equation*}
\Gamma(t):= \left(\frac{\nu\lambda_N}{2}-L_G^2\right) t-\frac{1}{\nu}\int_0^t\|u(s)\|^2_V\, {\rm d}s 
\end{equation*}
and we rewrite \eqref{sti_0} as
\begin{equation*}
{\rm d}\|r(t)\|^2_H + \left(\frac{\nu\lambda_N}{2} \|r(t)\|^2_H+ \Gamma^\prime(t)\|r(t)\|^2_H\right)  {\rm d}t
 \le {\rm d}M(t).
\end{equation*}
Multiplying  both members of the above expression by $e^{\Gamma(t)}$ and noticing that 
${\rm d}(e^{\Gamma(t)} \|r(t)|^2_H)
=
e^{\Gamma(t)} {\rm d}\|r(t)\|^2_H+ \Gamma^\prime(t)e^{\Gamma(t)} \|r(t)\|^2_H$, we get
\begin{equation*}
{\rm d}\left(e^{\Gamma(t)}\|r(t)\|^2_H\right)+ \frac{\nu \lambda_N}{2}e^{\Gamma(t)}\|r(t)\|^2_H {\rm d}t
\le e^{\Gamma(t)}{\rm d}M(t).
\end{equation*}
Integrating  in time  this bound up to a stopping time $\tau$ and taking the expected value we infer 
\[
\mathbb E \left[e^{\Gamma(t\wedge \tau)} \|r(t\wedge \tau)\|_H^2\right]
+\frac{\nu\lambda_N}{2} \mathbb E  \int_0^{t \wedge \tau} e^{\Gamma(s)} \|r(s)\|_H^2 {\rm d}s
\le 
\|r(0)\|_H^2 .
\]
This is \eqref{FP_est}.
\end{proof}

In order to control the integrating factor that appears in \eqref{FP_est}, we will make a suitable choice of the stopping time. For $R, \beta>0$, let
\begin{equation}
\label{tau_R_beta}
\tau_{R, \beta} := 
\inf \left\{r \ge 0: \frac{1}{\nu} \int_0^r \|u(s)\|^2_V\, {\rm d}s +\left(L_G^2-\frac{\nu \lambda_N}{4}\right)r - \beta \ge R\right\}
\end{equation} 
and $\tau_{R, \beta}=+\infty$ if the  set is empty, i.e. if 
\[
\frac{1}{\nu} \int_0^t \|u(s)\|^2_V\, {\rm d}s +\left(L_G^2-\frac{\nu \lambda_N}{4}\right)t - \beta < R
\qquad \forall t\ge 0.
\]

Here $N$ is the the parameter  of the finite dimensional control that 
 appears in the nudged equation \eqref{NS_abs_nud}, see Notation \ref{notation_N}.
The parameter $\beta$ will be useful to track the dependence on the initial data $u_0, v_0$ in subsequent estimates on $\tau_{R, \beta}$, see Propositions \ref{tau_exp}, \ref{tau_pol_ii} and \ref{tau_pol_iii}.

From the definition of $\tau_{R,\beta}$ in \eqref{tau_R_beta} we immediately get the following corollary of Proposition \ref{FP_thm}.
%COROLLARY
\begin{corollary}
\label{cor_tau_R_beta}
Under the same conditions as the Proposition \ref{FP_thm}, for any $u_0, v_0 \in H$ and any $R, \beta\ge 0$
\begin{equation*}
\mathbb{E} \left[{\pmb 1}_{(\tau_{R, \beta}=\infty)}\|u(t)-v(t)\|^2_H\right] \le e^{R+ \beta -\frac{\nu \lambda_N}{4}t}\|u_0-v_0\|^2_H.
\end{equation*}
\end{corollary}
%PROOF
\begin{proof}
It is enough to remark that if $\tau_{R, \beta}=\infty$, then 
$\frac{\nu \lambda_N}{4} t-\beta-R\le \Gamma(t)$ for any $t\ge 0$.
\end{proof}

\subsection{Decay estimates}
\label{decay_sec}
Let $\tau_{R, \beta}$ be the stopping time defined in \eqref{tau_R_beta}. In this Section we estimate the probability $\mathbb{P}(\tau_{R, \beta}<\infty)$ in terms of the parameter $R$. 
%\textcolor{red}{TAGLIA: Combing this estimate with Corollary \ref{cor_tau_R_beta} we will be able to infer the uniquess of the invariant measure (see Subsection \ref{uniq_sec}) and the mixing property (see Section \ref{mix_sec}).}
Under Assumption (G2)(i) we obtain an exponential decay in $R$ (Proposition \ref{tau_exp}), whereas under Assumption (G2) either (ii) or (iii), we obtain a polynomial decay in $R$ (Propositions \ref{tau_pol_ii} and \ref{tau_pol_iii}). 
%The different decay in $R$ will eventually lead to different type of mixing: exponential in case (G2)(i) and polynomial in cases (G2)(ii)-(iii) (see Section \ref{mix_sec}).

\begin{proposition}
\label{tau_exp}
Assume (G1) and (G2)(i). 
Consider the stopping time $\tau_{R, \beta}$ defined in \eqref{tau_R_beta}, 
where $u$ is the solution of the Navier-Stokes equation \eqref{NS_abs}.
  If
\begin{equation}
\label{cond_beta_i}
\beta\ge\frac{2}{\nu^2} \|u_0\|^2_H,
\end{equation}
then there exists  a positive integer $\bar N=\bar N(L_G, K_1, \nu, \|f\|_{V^*})$ such that for  any 
$N \ge \bar N$  we have 
\begin{equation}
\label{tau_est_i}
\mathbb{P}(\tau_{R, \beta}< \infty) 
\le 
e^{-CR},
\end{equation}
where $C=C(\lambda_1,\nu, K_1)$ is a positive constant independent of $R$, $\beta$ and $u_0$.
\end{proposition}
%PROOF
\begin{proof}
Keeping in mind the definition \eqref{tau_R_beta} of the stopping time $\tau_{R, \beta}$,  we introduce the set
\begin{equation}
\label{A}
A_{R, \beta}= \left\{\sup_{r \ge 0} \left[\frac{1}{\nu} \int_0^r \|u(s)\|^2_V\, {\rm d}s +\left(L_G^2-\frac{\nu \lambda_N}{4}\right)r - \beta\right] \ge R\right\}
\end{equation}
so that $\mathbb{P}(\tau_{R, \beta}< \infty) \le \mathbb{P}(A_{R, \beta})$. 
Thus we need to estimate $\mathbb{P}(A_{R, \beta})$.

Its complementary set can be written as follows
\begin{equation}
\label{A^C}
A_{R, \beta}^c=\left\{ \frac \nu 2  \int_0^r\|u(s)\|^2_V\, {\rm d}s < \frac{\nu^2}2 \left[\left(\frac{\nu \lambda_N}{4}-L_G^2\right)r+\beta +R\right] \text{ for any }\ r \ge 0
\right\}.
\end{equation}
We take $\bar N>0$ large enough such that 
\begin{equation}
%\label{condition_lambda_N_i}
\frac{\nu^2}2 \left( \frac{\nu\lambda_{\bar N}}{4}-L_G^2\right)>K_1^2+ \frac{1}{\nu}\|f\|^2_{V^*}.
\end{equation}
We recall that $K_1$ is the constant appearing in Assumption (G2)(i).
Choosing $\beta$ as in \eqref{cond_beta_i} and setting 
$
\bar R:=\frac{\nu^2}2 R$, for any   $N \ge \bar N$
 we get 
\begin{equation*}
A_{R, \beta}^c 
\supseteq 
\left\{ \frac \nu 2 \int_0^r \|u(s)\|^2_V\, {\rm d}s <\left( K_1^2+ \frac{1}{\nu} \|f\|^2_{V^*} \right)r+ \bar R + \|u_0\|^2_H  \text{ for any } r \ge 0\right\}
\end{equation*}
i.e.
\[
A_{R, \beta} \subseteq 
\left\{ \sup_{r\ge 0} \left[ \frac \nu 2 \int_0^r \|u(s)\|^2_V\, {\rm d}s -\left( K_1^2+ \frac{1}{\nu} \|f\|^2_{V^*}\right)r-\|u_0\|^2_H \right] \ge \bar R\right\} .
\]
From \eqref{P_est_exp} in Proposition \ref{lem_exp} we therefore conclude that 
\begin{align*}
 \mathbb{P}(A_{R, \beta}) \le 
 e^{-\frac{\nu\lambda_1}{8K_1^2}\bar R}.
 \end{align*}
 Since $\mathbb{P}(\tau_{R, \beta}< \infty) \le \mathbb{P}(A_{R, \beta})$, 
 keeping in mind the definition of $\bar R$ the estimate  \eqref{tau_est_i} immediately follows.
\end{proof}

\begin{proposition}
\label{tau_pol_ii}
Assume (G1) and (G2)(ii). 
Consider the stopping time $\tau_{R, \beta}$ defined in \eqref{tau_R_beta}, 
where $u$ is the solution of the Navier-Stokes equation \eqref{NS_abs}.
  If
\begin{equation}
\label{cond_beta_ii}
  \beta \ge \frac{1}{\nu^2} (C_b+\|u_0\|^2_H),
\end{equation}
with $C_b$ the constant that appears in estimate \eqref{P_est_ii},
then  there exists  a positive integer $\bar N=\bar N(\nu, L_G, K_2, \tilde{K}_2, \lambda_1, \gamma, \|f\|_{V^*})$ 
 such that for  any $N \ge \bar N$  we have 
\begin{equation}
\label{tau_est_ii}
\mathbb{P}(\tau_{R, \beta}< \infty) 
\le 
 \frac{C(1+ \|u_0\|_H^{4(p+1)})}{R^p},
 \end{equation}
 for any $p>0$,  where $C=C(\lambda_1,p, \nu, K_2, \tilde{K}_2, \gamma, \|f\|_{V^*})$
  is a positive constant  independent of $R, \beta$ and $u_0$.
 \end{proposition}
 %PROOF
 \begin{proof}
The proof follows the line of the proof of Proposition \ref{tau_exp}. Consider the set $A_{R, \beta}$ and its complementary  set $A_{R, \beta}^c$ introduced in \eqref{A} and \eqref{A^C}, respectively.
We take $\bar N>0$ large enough such that 
\begin{equation}
%\label{condition_lambda_N_ii}
  \nu^2  \left( \frac{\nu\lambda_{\bar N}}{4}-L_G^2\right)> C_b,
\end{equation}
where $C_b=C_b(K_2, \tilde{K}_2, \lambda_1, \nu, \gamma, \|f\|_{V^*})$ is the constant appearing in \eqref{P_est_ii} of Proposition \ref{lem_pol}.
Choosing $\beta$ as in \eqref{cond_beta_ii} and setting 
$\bar R:=\frac{\nu^2R}{2}$, for any $N \ge \bar N$  we get 
\[
A_{R, \beta}^c \supseteq \left\{ \nu  \int_0^r \|u(s)\|^2_V\, {\rm d}s -C_b(r+1)< \bar R + \|u_0\|^2_H,  \text{ for any } r \ge 0\right\}
\]
i.e.
\[
A_{R, \beta} \subseteq 
\left\{ \sup_{r\ge 0} \left[\nu \int_0^r \|u(s)\|^2_V\, {\rm d}s -C_b(r+1)-\|u_0\|^2_H \right] \ge \bar R\right\} .
\]
From \eqref{P_est_ii} we therefore conclude that, for any $q>2$, 
\begin{align*}
 \mathbb{P}(A_{R, \beta}) \le 
 \frac{C(1+\|u_0\|_H^{2q})}{\bar R^{\frac q2-1}}
 \end{align*}
 where $C=C(\lambda_1, q, \nu, K_2, \tilde{K}_2, \gamma, \|f\|_{V^*})$.
 Since $\mathbb{P}(\tau_{R, \beta}< \infty) \le \mathbb{P}(A_{R, \beta})$, 
 the estimate \eqref{tau_est_ii} immediately follows.
 \end{proof}

 \begin{proposition}
 \label{tau_pol_iii}
Assume (G1), (G2)(iii) and
\begin{equation}\label{condizione32}
\nu >  \frac{3\tilde{K}_3^2}{2\lambda_1}.
\end{equation}
Consider the stopping time $\tau_{R, \beta}$ defined in \eqref{tau_R_beta}, 
where $u$ is the solution of the Navier-Stokes equation \eqref{NS_abs}.
If
\begin{equation}
\label{cond_beta}
\beta \ge \frac{C_b+\|u_0\|^2_H} {\nu(\nu- \frac {\tilde{K_3}^2}{2\lambda_1})} ,
\end{equation}
with $C_b$ the constant that appears in \eqref{P_est_iii}, 
then   there exists  a positive integer $\bar N=\bar N(\nu, L_G, K_3,\tilde{K_3}, \lambda_1, \|f\|_{V^*})$  such that for  any $N \ge \bar N$ we have 
\begin{equation}
\label{tau_est_iii}
\mathbb{P}(\tau_{R, \beta}< \infty) 
\le 
 \frac{C(1+ \|u_0\|_H^{4(p+1)})}{R^{p}}
\end{equation}
for any $p\in \left(0,\frac{\nu \lambda_1}{2 \tilde K_3^2}-\frac34\right)$, 
where $C=C(\lambda_1,q,\nu, K_3, \tilde{K_3}, \|f\|_{V^*})>0$  is a positive constant  independent of $R$, $\beta$ and $u_0$.
\end{proposition}
%PROOF
\begin{proof}
We introduce the set $A_{R, \beta}$ as in \eqref{A} and write its complementary set as follows
\begin{equation*}
A_{R, \beta}^c
=
\left\{\left(\nu-\frac{\tilde{K}_3^2}{2\lambda_1} \right) \int_0^r\|u(s)\|^2_V\, {\rm d}s 
\le
\nu \left(\nu-\frac{\tilde{K}_3^2}{2\lambda_1} \right)\left[\left(\frac{\nu \lambda_N}{4}-L_G^2\right)r+\beta +R\right], \text{ for all } r \ge 0
\right\}.
\end{equation*}
We take $\bar N>0$ large enough such that 
\begin{equation}
%\label{condition_lambda_N}
\nu\left(\nu-\frac{\tilde{K}_3^2}{2\lambda_1}\right)\left( \frac{\nu\lambda_{\bar N}}{4}-L_G^2\right)>C_b,
\end{equation}
with $C_b$ the constant appearing in \eqref{P_est_iii}. Choosing $\beta$ as in \eqref{cond_beta} and setting 
\begin{equation}
\label{R_bar}
\bar R:=\nu\left(\nu-\frac{\tilde{K}_3^2}{2\lambda_1}\right)R,
\end{equation} 
for any   $N\ge \bar N$ we get  
\begin{equation*}
A_{R, \beta}^c \supseteq \left\{ \left(\nu-\frac{\tilde{K}_3^2}{2\lambda_1}\right) \int_0^r \|u(s)\|^2_V\, {\rm d}s -C_b(r+1)\le \bar R + \|u_0\|^2_H, \text{ for all }  r \ge 0\right\}.
\end{equation*}
From \eqref{P_est_iii}, provided $\nu> \frac{3\tilde{K}_3^2}{2\lambda_1}$ and $2<q<\frac 12 +\frac{\nu \lambda_1}{\tilde{K}_3^2}$, we therefore conclude that 
\begin{align*}
\mathbb{P}(\tau_{R, \beta}< \infty) \le \mathbb{P}(A_{R, \beta}) \le 
 \frac{C(1+ \|u_0\|_H^{2q})}{\bar R^{\frac q2-1}},
 \end{align*}
where $C=C(\lambda_1,q,\nu, K_3, \tilde{K_3}, \|f\|_{V^*})>0$ is a positive constant independent of $\bar R$. By taking $p=\frac q2 -1$ and keeping in mind \eqref{R_bar}, the estimate \eqref{tau_est_iii} immediately follows with the power $p= \frac q2 -1\in(0,\frac{\nu\lambda_1}{2 \tilde{K}_3^2}-\frac 34)$, since $2<q<\frac 12 +\frac{\nu \lambda_1}{\tilde{K}_3^2}$.
\end{proof}

\begin{remark}
We emphasize the difference between Propositions \ref{tau_pol_ii} and \ref{tau_pol_iii}. In Proposition \ref{tau_pol_ii} we have a polynomial decay in $R$, with an arbitrary exponent $p>0$. In Proposition \ref{tau_pol_iii} the type of decay in $R$ is still polynomial but now the range of admissible exponents $p$ depends on the viscosity coefficient $\nu$ and the constant $\tilde{K}_3$ that, roughly speaking, represents the intensity of the multiplicative part of the noise. In particular, in this latter case, we need to impose the condition $\nu>\frac{3\tilde{K}_3^2}{2\lambda_1}$ on the viscosity coefficient to ensure the existence of an admissible set of exponents $p$. 
%\begin{footnote}{\textcolor{blue}{Si riesca ad indebolire l'assunzione richiedendo $\nu > \frac{\tilde{K}_3^2}{2\lambda_1}$, cos\'i da lavorare sotto le stesse ipotesi che ci danno esistenza di una misura invariante?}}\end{footnote}
\end{remark}

\subsection{The Foias-Prodi estimates}

As a consequence of  Proposition \ref{FP_thm} and Propositions \ref{tau_exp}, \ref{tau_pol_ii} and \ref{tau_pol_iii} we now show that,
provided $N$ is taken sufficiently large, $\mathbb{E}\left[ \|u(t)-v(t)\|^2_H\right]$
vanishes as $t\to+\infty$. 
The convergence rate depends on the growth of $G$ as specified by the three different assumptions (G2).

%THEOREM
\begin{theorem}[Foias-Prodi estimates]
\label{FP_cor}
Assume (G1) and  $u_0, v_0 \in H$. Let 
$u$ be the solution of the Navier-Stokes equation \eqref{NS_abs} and $v$ that of its nudged equation\eqref{NS_abs_nud} with $\lambda=\frac{\nu \lambda_N}2$.
\begin{enumerate}[label=$(\roman{*})$]
%a
\item%[a.] 
If (G2)(i) holds, then there exists a positive integer
$\bar N=\bar N(L_G, K_1, \nu, \|f\|_{V^*})$  and positive constants $C$ and $\delta$ such that for any  $N \ge \bar N$ we have 
\begin{equation}
%\label{exp_dec_est}
\mathbb{E}\left[ \|u(t)-v(t)\|^2_H \right] \le 
C(1+ \|u_0\|^2_H+ \|v_0\|^2_H) \left(1+e^{\frac{2}{\nu^2}\|u_0\|^2_H}%\pmb{1}_{(u_0 \ne v_0)}
\right)e^{-\delta t}\qquad \forall t>0.
\end{equation}
Here  $C$ and $\delta$ do not depend on $\|u_0\|_H, \|v_0\|_H$ and $t$
%b
\item%[b.] 
If (G2)(ii) holds, then there exists a positive integer 
$\bar N=\bar N(L_G, K_2,\tilde{K}_2, \nu, \lambda_1, \gamma, \|f\|_{V^*})$ 
and positive constants $C$ and $\alpha$ such that  for any  $N \ge \bar N$ we have 
\begin{equation}
\label{pol_dec_est_ii}
\mathbb{E}\left[ \|u(t)-v(t)\|^2_H \right] 
\le\frac{C}{t^p}(1+ \|u_0\|_H^2+\|v_0\|^2_H)\left(1%+ \|u_0\|_H^{2(2p+1)}\pmb{1}_{(u_0=v_0)}
+e^{\alpha \|u_0\|_H^2}   %\pmb{1}_{(u_0 \ne v_0)}
\right)\qquad \forall t>0,
\end{equation}
 where $p$ is any positive number. Here 
$C$ and $\alpha$ do not depend on $\|u_0\|_H, \|v_0\|_H$ and $t$ but depend on $p$.
%c
\item%[c.]  
If (G2)(iii)  and   \eqref{condizione32} hold, 
then there exists a positive integer 
$\bar N=\bar N(L_G, K_3, \tilde{K}_3, \nu, \lambda_1, \|f\|_{V^*})$ and 
 positive constants $C$ and $\alpha$ 
such that  for any  $N \ge \bar N$
 the estimate \eqref{pol_dec_est_ii} holds for any 
$p\in (0,\frac{\nu \lambda_1}{ 4 \tilde{K}_3^2}-\frac 38)$.
\end{enumerate}
\end{theorem}

\begin{proof}
In all the cases (G2)(i), (G2)(ii) and (G2)(iii) the structure of the proof is the same. We therefore prove all the statements in a unified way.

Let $u$ and $v$ be the solutions to \eqref{NS_abs} and \eqref{NS_abs_nud} starting from $u_0, v_0 \in H$ respectively. 
By means of the H\"older and the Young inequalities, invoking Corollary \ref{cor_tau_R_beta} and estimates \eqref{sup_q_est} and \eqref{est_v} with $q=4$ \begin{footnote}{Notice that considering $q=4$ in  \eqref{sup_q_est} and \eqref{est_v} requires to impose the condition $1+ \frac{2\nu \lambda_1}{\tilde{K}_3^2}>4$, equivalent to $ \nu >\frac{3\tilde K_3^2}{2\lambda_1}$.}\end{footnote}, we infer 
\begin{align}
\label{sti_diff_proof}
\mathbb{E} \left[ \|u(t)-v(t)\|^2_H\right]
&= \mathbb{E} \left[\pmb{1}_{(\tau_{R, \beta =+\infty ) }} \|u(t)-v(t)\|^2_H\right]+ \mathbb{E} \left[\pmb{1}_{(\tau_{R, \beta <+\infty)}} \|u(t)-v(t)\|^2_H\right]
\notag \\
& \le e^{\beta+ R - \frac{\nu \lambda_N}{4}t }\|u_0-v_0\|^2_H+ \left( \mathbb{P}(\tau_{R, \beta}<\infty)\right)^{\frac 12} \left(\mathbb{E}[\|u(t)-v(t)\|^4_H \right)^{\frac 12}
\notag\\
& \le
 C\left(1+\|u_0\|^2_H+ \|v_0\|^2_H \right)\left(\left( \mathbb{P}(\tau_{R, \beta}<\infty)\right)^{\frac 12}+ e^{\beta+ R - \frac{\nu \lambda_N}{4} t}\right),
\end{align}
where $C$ is a positive constant that depends on the parameters of equations \eqref{NS_abs} and \eqref{NS_abs_nud} (see Lemmata \ref{lem_A2} and \ref{lem_v} for the explicit dependence, according to which assumption (G2) we make on the noise) and is independent of $R, \beta, u_0$ and $v_0$.
Now we use the previous bounds on $\tau_{R, \beta}$; by  Propositions \ref{tau_exp}, \ref{tau_pol_ii} and \ref{tau_pol_iii}
for suitably chosen $\beta$ (see \eqref{cond_beta_i}, \eqref{cond_beta_ii}, \eqref{cond_beta}) 
we get  for any $R>0$
\begin{equation*}
\mathbb{P}(\tau_{R, \beta} < \infty)^{\frac 12} \le
\begin{cases}
e^{-CR} & \text{under (G2)(i)}, \quad \text{with $C=C(\lambda_1, \nu, K_1)$}
\\
\frac{C\left(1+\|u_0\|^{2(p+1)}_H\right)}{R^{\frac p2}}, & \text{ for any $p>0$,  under (G2)(ii)}, \quad \text{with $C=C\left(\lambda_1, p, \nu, K_2, \tilde{K}_2,\gamma, \|f\|_{V^*}\right)$}
\\
\frac{C\left(1+\|u_0\|^{2(p+1)}_H\right)}{R^{\frac p2}}, & \text{ for any $p\in (0,\frac{\nu \lambda_1}{2\tilde{K}_3^2}-\frac 34)$ \ under (G2)(iii)}, \quad \text{with $C=C\left(\lambda_1, p, \nu, K_3, \tilde{K}_3, \|f\|_{V^*}\right)$}
\end{cases}
\end{equation*}
where we emphasize that the constants $C$ that appear in the above expressions do not depend on $u_0, v_0, R, \beta$ and $t$. Coming back to estimate \eqref{sti_diff_proof}, 
if we select $R= \frac{\nu \lambda_N}{8}t$, for each $t > 0$ and take $\beta$ according to the lower bounds in  \eqref{cond_beta_i}, \eqref{cond_beta_ii}, \eqref{cond_beta}, we conclude the proof. 
In the cases (G2)(ii) and (G2)(iii) the polynomial dependence on $\|u_0\|_H$ is estimated by an exponential function.
\end{proof}

\section{Ergodic results}
\label{uniq_inv_mea_sec}
In this Section we prove the existence and uniqueness of the invariant measure for \eqref{NS_abs} and the weak convergence to it, also named asymptotic stability of the invariant measure.  As anticipated in Theorem \ref{thm_intro}, in the study of the long time behavior of the solution, working under Assumptions 
(G2)(i) or (G2)(ii) does not require any restriction on the viscosity coefficient $\nu$; we will in fact prove that in these cases at least one invariant measure always exists and it is unique and asymptotically stable under a non-degeneracy condition on the noise.
Things are more delicate under (G2)(iii): the existence of  invariant measures, their uniqueness and the 
asymptotic stability require gradually narrower assumptions about the viscosity coefficient $\nu$. Dissipation is required to balance the intensity of the multiplicative part of the noise $\tilde K_3$ more and more consistently.
Therefore, for clarity of exposition, we separate the results of existence, uniqueness to asymptotic stability of the invariant measure by dividing them into the three Sections \ref{exis_inv_meas}, \ref{main_result_sec} and \ref{asy_sta_sec}. The existence result is well known in the literature (see \cite{Fla_Gat}) but we briefly recall it. The uniqueness result and the asymptotic stability result are based on the abstract results of \cite{GHMR17} and \cite{KS} respectively; we recall them in Section \ref{exis_inv_meas}.

\subsection{Existence of an invariant measure}
\label{exis_inv_meas}

For every $x\in H$, the unique solution 
to equation \eqref{NS_abs} as given in Proposition~\ref{ex_uniq_sol} will be denoted by $u(\cdot; x)$,
and for every $t\in[0,T]$ we set  $u(t;x)$ for its value at time $t$, and $u(t;x):\Omega\to H$ is a random variable in $L^2(\Omega, \mathcal{F}_t; H)$.

We denote by 
$\mathfrak{B}(H)$ the $\sigma$-algebra of all Borel subsets of $H$ and by $\mathcal{P}(H)$ the set of all probability measures on $(H, \mathfrak{B}(H))$. 
Also, the symbol $\mathcal{B}_b(H)$ denotes the space of Borel measurable bounded functions from $H$ to $\mathbb{R}$ and $\mathcal{C}_b(H)$ the space of continuous bounded functions from $H$ to $\mathbb{R}$. 

With this notation and by virtue of Theorem~\ref{ex_uniq_sol}, we can introduce the Markov kernel
\begin{equation}
%\label{P_t_kernel}
P_t(x, A):=\mathbb{P}(u(t;x)\in A), \quad \forall t \ge 0, \ x \in H, \ A \in \mathfrak{B}(H).
\end{equation}
This kernel defines a family of operators $P:=(P_t)_{t\ge 0}$ that 
act on functions $\varphi \in \mathcal{B}_b(H)$ as 
\begin{equation}
%\label{P_t}
(P_t\varphi)(x):=\int_H \varphi (y)P_t(x, {\rm d}y)= \mathbb{E}[ \varphi(u(t;x))], \quad x\in H, \ t \ge 0.
\end{equation}
For any Borel probability measures $\mu \in \mathcal{P}(H)$ we consider the evolution of measures
\begin{equation*}\label{P_mu}
P^*_t\mu(A):=\int_H P_t(y,A)\, \mu({\rm d}y), \quad  A \in \mathfrak{B}(H), \ t \ge 0.
\end{equation*} 
%\textcolor{red}{\textbf{Previous version (al posto delle 6 righe scritte sopra): da tagliare.}With this notation and by virtue of Theorem~\ref{ex_uniq_sol}, we can introduce the family of operators$P:=(P_t)_{t \ge 0}$ associated to equation \eqref{NS_abs} as \begin{equation}\label{P_t}(P_t\varphi)(x):= \mathbb{E}[ \varphi(u(t;x))], \quad x\in H ,\quad \varphi\in \mathcal{B}_b(H).\end{equation}}
It is clear that $P_t\varphi$ is bounded 
for every $\varphi \in \mathcal{B}_b(H)$. We know from \cite[Corollary 23]{On2005}
that the transition function is jointly measurable, that is for any $A\in\mathfrak{B}(H)$ 
the map $H \times [0,\infty)\ni  (x,t)\mapsto P_t(x,A)\in \mathbb{R}$ is measurable. So $P_t\varphi$ is also measurable for every $\varphi \in \mathcal{B}_b(H)$, 
hence $P_t$ maps $\mathcal{B}_b(H)$ into itself for every $t \ge0$.
Furthermore, since the unique solution of \eqref{NS_abs} is an $H$-valued continuous process, 
then it is also a Markov process, see \cite[Theorem 27]{On2005}. 
Therefore we deduce that the family of operators $(P_t)_{t\ge0}$ is a Markov semigroup, 
namely $P_{t+s}=P_tP_s$ for any $s,t\ge0$.
\\
We are ready to give the precise definition of invariant measure. We recall that a semigroup $P$ is said to be Feller if $P_t: \mathcal{C}_b(H) \rightarrow \mathcal{C}_b(H)$, for all $t> 0$.
\begin{definition}
Given a Feller semigroup $P$ an invariant measure for $P$ 
  is a probability measure $\mu\in\mathcal{P}(H)$ such that $P^*_t\mu=\mu$ for all $t\ge 0$ or, equivalently, 
  \[
  \int_H \varphi(x)\,\mu({\rm d} x) = \int_H P_t\varphi(x)\,\mu({\rm d} x) \quad\forall\,t\geq0,\quad\forall\,\varphi\in \mathcal{C}_b(H).
  \]
\end{definition}
The following result shows that the transition semigroup $P$
of equation \eqref{NS_abs} admits at least one invariant measure. 
\begin{proposition}
\label{prop_mis_inv}
Assume (G1) and (G2) with the additional condition 
\begin{equation}\label{visco1}
\nu>\frac{\tilde{K}_3^2}{2\lambda_1}
\end{equation}
for the case (G2)(iii). 
Then, the transition semigroup $P$ admits at least one invariant measure.
\end{proposition}
\begin{proof}
The result is a consequence of the Krylov-Bougoliubov Theorem (see e.g. \cite[Theorem~11.7]{DapZab}) provided 
that we check that $P$ is Feller and the tightness property holds.\\
(i) Let us show at first that $P$ is Feller. Let $t>0$ and $\varphi \in \mathcal{C}_b(H)$ be fixed. We need to prove that, given a sequence $\{x_n\}_n \subset H$ which converges in $H$ to $x\in H$ as $n \rightarrow \infty$, the sequence $P_t\varphi(x_n)$ converges to $P_t \varphi(x)$ as $n \rightarrow \infty$. Lemma \ref{techn_lemm} yields 
\begin{equation*}
\mathbb{E}\left[ e^{-\left(L_G^2t- \lambda_1 \nu t+\frac{1}{\nu}\int_0^t\|u(s;x)\|^2_V\, {\rm d}s \right)}\|u(t;x)-u(t;x_n)\|^2_H\right] \le \|x-x_n\|_H^2.
\end{equation*}
It follows that $u(t;x_n)$ converges to $u(t;x)$ in probability. This implies, by the continuity of $\varphi$, that $\varphi(u(t;x_n))$ converges to $\varphi(u(t;x))$ in probability. The boundedness of $\varphi$ and the Vitali Theorem yield, in particular, $\varphi(u(t;x_n))\rightarrow\varphi(u(t;x))$ in $L^1(\Omega)$ and thus
\begin{equation*}
|P_t\varphi(x_n)-P_t\varphi(x)| \le \mathbb{E}\left[\left|\varphi(u(t;x_n))-\varphi(u(t;x)) \right| \right] \rightarrow 0, \quad \text{as} \ n \rightarrow + \infty,
\end{equation*}
which proves the Feller property in $H$.
\\
(ii) We prove now that $P$ satisfies the tightness property of the Krylov-Bougoliubov Theorem.
We use the estimate of Lemma \ref{lem_A1}. 
To this end, let $x=0$. We are going to show that the family of measures 
$(\mu_t)_{t>0}\subset \mathcal{P}(H)$ defined by 
\begin{equation*}
\mu_t : A \mapsto \frac 1t \int_0^t (P_s\pmb{1}_{A})(0)\, {\rm d}s
=\frac 1t \int_0^t P_t(0,A)\, {\rm d}s, \qquad A \in \mathcal{B}(H), \ t>0,
\end{equation*}
is tight in $H$.
Let $B_n$ be the closed ball in $V$ of radius $n \in \mathbb{N}$, $B_n$ is a compact subset of $H$, since the embedding $V \hookrightarrow H$ is compact. 
Hence, Lemma~\ref{lem_A1} and the Chebychev inequality yield, for any $t>0$,
\begin{align*}
\mu_t(B_n^c)
&=\frac 1t \int_0^t (P_s\pmb{1}_{B_n^c})(0)\, {\rm d}s=
\frac 1t \int_0^t \mathbb{P}\left(\|u(s;0)\|^2_V \ge n^2 \right)\, {\rm d}s
\\
&\le \frac{1}{tn^2}\int_0^t \mathbb{E}\left[ \|u(s;0)\|^2_V\right]\, {\rm d}s
\le \frac{1}{n^2}\frac ba,
\end{align*}
with $a$ and $b$ defined in \eqref{a} and \eqref{b} respectively,
from which
\[
\forall \varepsilon >0 \ \exists \  n_\varepsilon: \mu_t(B_{n_\varepsilon})>1- \varepsilon
\qquad \text{ for any } t\ge 0
\]
and the thesis follows. 
\end{proof}

\begin{remark}
The condition $\nu>\frac{\tilde K_3^2}{ 2 \lambda_1}$ that appears in Proposition \ref{prop_mis_inv}, when we work under Assumption (G2)(iii), roughly speaking says that the viscosity coefficient has to balance the intensity of the \emph{multiplicative} part of the noise. Notice that the same condition appears in \cite[Theorem 4.1]{Fla_Gat}; compare also with the similar condition (3.6) that appears in \cite[Theorem 3.3.]{noi} in the case of the nonlinear Schr\"odinger equation.
% Notice that in the case of a bounded noise, that is when $\tilde K_1=0$, the previous condition reduces to $\nu>0$. 
\end{remark}

\subsection{The abstract results in \cite{GHMR17} and \cite{KS}}
\label{KS_sec}

The proof of the uniqueness of the invariant measure and its asymptotic stability relies on the abstract results of \cite{GHMR17} and \cite{KS}, respectively.
In \cite{GHMR17} the authors provide sufficient conditions in terms of generalized couplings for the uniqueness of an invariant measure, whereas in \cite{KS}
sufficient conditions for the weak convergence to the invariant measure 
are provided under more restrictive assumptions. 
The statement of these results requires to fix some notation.

We work on the Polish space $H$ and on it we consider the metric induced by the norm $\|\cdot\|_H$. This metric induces on the space $\mathcal{P}(H)$ % of probability measures on $H$ 
the weak convergence% that we denote by $\rightharpoonup$. 
: $\{\mu_k\}_k \subset \mathcal{P}(H)$ weakly converges to $\mu\in \mathcal{P}(H)$ if 
\begin{equation*}
\int_H f\, {\rm d}\mu_k \rightarrow \int_H f\,{\rm d}\mu \quad \text{as} \ k \rightarrow \infty, \quad \forall \ f \in C_b(H).
\end{equation*}
We denote by $\text{Lip}_b(H)$ the space of all bounded and Lipschitz real-valued functions on $H$, endowed with the norm 
\begin{equation*}
\|\varphi\|_L= \sup_{x\in H}|\varphi(x)|+\sup_{x, y \in H, x\neq y}\frac{|\varphi(x)-\varphi(y)|}{\|x-y\|_H}.
\end{equation*} 
We endow the space $\mathcal{P}(H)$ with the Wasserstein (also called dual-Lipschitz) distance 
\begin{equation}
%\label{Wass}
\|\mu- \nu\|_* := \sup_{\varphi \in \text{Lip}_b(H), \|\varphi\|_L \le 1}\left| \int_H\varphi\,{\rm d}\mu-\int_H\varphi\,{\rm d}\nu\right|, \quad \mu, \nu \in \mathcal{P}(H).
\end{equation}
From \cite[Theorem 1.2.15]{KS_BOOK} we know that a sequence $\{\mu_k\}_k \subset \mathcal{P}(H)$ converges to a measure $\mu \in \mathcal{P}(H)$ w.r.t. the Wasserstein distance if and only if $\{\mu_k\}_k$ weakly converges to $\mu$.

All the main statements below will be formulated in the discrete-time setting; however, they have straightforward analogues in the continuous-time setting  thanks to the continuity of the trajectories of the solution (see, e.g.,  \cite[Remark 4]{KS}).

We introduce the space of one-sided infinite sequences $H^{\mathbb{N}}$ with its Borel $\sigma$-field $\mathfrak{B}(H^{\mathbb{N}})$. By $\mathcal{P}\left(H^{\mathbb{N}}\right)$ we denote the collections of Borel probability measures on $H^{\mathbb{N}}$.

For given $\mu, \nu \in \mathcal{P}(H^{\mathbb{N}})$ we define
\begin{equation*}
\mathcal{C}(\mu, \nu):= \{ \xi \in \mathcal{P}(H^{\mathbb{N}} \times H^{\mathbb{N}})\ : \ \pi_1(\xi)= \mu, \ \pi_2(\xi)= \nu\},
\end{equation*}
where $\pi_i(\xi)$ denotes the $i$-th marginal distribution of $\xi$, $i=1,2$. Any $\xi \in C(\mu, \nu)$ is called a \textit{coupling} for $\mu, \nu$. We introduce the following two extensions of the notion of coupling. Recall that $\mu \ll \nu$ means that $\mu$ is absolutely continuous w.r.t. $\nu$ and $\mu \sim \nu$ means that $\mu$ and $\nu$ are equivalent, i.e., mutually absolutely continuous. We define 
\begin{equation*}
\widetilde{\mathcal{C}}
(\mu, \nu):= \{ \xi \in \mathcal{P}(H^{\mathbb{N}} \times H^{\mathbb{N}})\ : \ \pi_1(\xi) \sim \mu, \ \pi_2(\xi) \sim \nu\},
\end{equation*}
\begin{equation*}
\widehat{\mathcal{C}}
(\mu, \nu):= \{ \xi \in \mathcal{P}(H^{\mathbb{N}} \times H^{\mathbb{N}})\ : \ \pi_1(\xi) \ll \mu, \ \pi_2(\xi) \ll \nu\}
\end{equation*}
and call any probability measure from the classes $\widetilde{C}(\mu, \nu)$, $\widehat{C}(\mu, \nu)$ a \textit{generalized coupling} for $\mu, \nu$.

We introduce the subspaces
\begin{equation*}
D:= \{(x,y) \in H^{\mathbb{N}} \times H^{\mathbb{N}} \ : \ \lim_{n \rightarrow \infty}\| x(n)-y(n)\|_H=0\}
\end{equation*}
and, for a given $\varepsilon>0$, $n \in \mathbb{N}$,
\begin{equation*}
D^n_{\varepsilon}:= \{(x,y) \in H^{\mathbb{N}} \times H^{\mathbb{N}} \ : \ \|x(n)-y(n)\|_H\le \varepsilon\}.
\end{equation*}
We also introduce the set of test functions 
\begin{equation}
\label{G}
\mathcal{G}= \left\{\varphi \in C_b(H) \ : \ \sup_{x \ne y}\frac{|\varphi(x)-\varphi(y)|}{ \|x-y\|_H}< \infty \right\}
\end{equation}
which is  determining  measure set in $H$, that is, if $\mu, \nu \in \mathcal{P}(H)$ are such that $\int_H \varphi(u) \mu({\rm d}u)=\int_H \varphi(u) \nu({\rm d}u)$ for all $\varphi \in \mathcal{G}$, then it follows that $\mu=\nu$.

Let $u$ be the unique solution to equation \eqref{NS_abs}, the law of the sequence $\{u(n)\}_{n \in \mathbb{N}}$ on $(H^{\mathbb{N}}, \mathfrak{B}(H^{\mathbb{N}}))$ with initial velocity $u_0$ will be denoted by $\mathbb{P}_{u_0}$.
We are now ready to state the abstract result from \cite{GHMR17} in the form that best fits our context. 

\begin{theorem}
\label{GHMR}
Suppose that $\mathcal{G}$ determines measures on $(H, \|\cdot\|_H)$, and that $D \subseteq H^{\mathbb{N}} \times H^{\mathbb{N}}$ is measurable. If, for each $u_0, v_0 \in H$, there exists a generalized coupling $\xi_{u_0,v_0}\in \widehat{\mathcal{C}}(\mathbb{P}_{u_0},\mathbb{P}_{v_0})$ such that $\xi_{u_0,v_0}(D)>0$, then there is at most one $P$-invariant probability measure $\mu \in \mathcal{P}(H)$.
\end{theorem}
Roughly speaking, the above result states that in order for the Markov semigroup to have at most one invariant measure one needs to ensure that the process couple \textit{asymptotically} on a set of positive probability.

%Any such measure $\Gamma_{u_0,v_0}\in \widetilde{\mathcal{C}}(\delta_{u_0}P^{\mathbb{N}},\delta_{v_0}P^{\mathbb{N}})$ for which $D_{\rho} \subseteq \mathcal{H}^{\mathbb{N}} \times \mathcal{H}^{\mathbb{N}}$ exists is called \textit{asymptotic generalized
%\begin{footnote}{\textcolor{blue}{Cos\'i \'e chiaro che non ho l'UGUAGLIANZA delle leggi ma l'ASSOLUTA CONTINUIT'A: vedi intro Butkowsky et all.}}\end{footnote}
% coupling} for $\delta_{u_0}P^{\mathbb{N}},\delta_{v_0}P^{\mathbb{N}}$.

Under more restrictive assumptions, the abstract result in \cite{KS} (see in particular Corollary 4) ensures the asymptotic stability of the invariant measure.
\begin{theorem}
\label{KS}
Suppose that  the transition semigroup $P$ associated to \eqref{NS_abs} is a Feller semigroup on $H$ and for any $u_0, v_0 \in H$ there exists some $\xi_{u_0, v_0} \in \widehat{C}(\mathbb{P}_{u_0}, \mathbb{P}_{v_0})$ such that $\pi_1(\xi_{u_0, v_0})\sim \mathbb{P}_{u_0}$ and for any $\varepsilon>0$
\begin{equation}
\label{hyp_KS}
\lim_{n \rightarrow \infty} \xi_{u_0,v_0}\left(D^n_\varepsilon\right)=1.
\end{equation}
Then there exists at most one invariant probability measure and, if such a measure $\mu$ exists, then 
\begin{equation*}
\|P_t^*\delta_{u_0}-\mu\|_* \rightarrow 0 \quad \text{as} \ n \rightarrow \infty, \quad \forall \ u_0 \in H.
\end{equation*}
%$P_n(u_0, \cdot) \rightharpoonup \mu$ for any $u_0 \in H$ as $n \rightarrow \infty$. 
\end{theorem}

\subsection{Uniqueness of the invariant measure}
\label{main_result_sec}

We prove here the uniqueness of the invariant measure relying on Theorem \ref{GHMR}. The key role in the proof is played by the Foias-Prodi type estimates that quantify the minimum number $\bar N$ of modes that need to be activated by the noise in order to get synchronization at infinity. We will thus need to impose the following requirement on the range of the noise: Assumption (G3) has to hold for some $M \ge \bar N$, where $\bar N$ is as in Theorem \ref{FP_cor}. 
%This means that the noise has to act on the sufficiently large number $\bar N$ of modes that ensures Foias-Prodi estimates from Theorem \ref{FP_cor} to hold. 
This condition resembles the one  usually assumed  in presence of an \textit{additive} noise (see the many examples in \cite{GHMR17}).
%We emphasize here that under Assumption (G2)(i) and (ii) we prove the uniqueness of the invariant measure and the asymptotic stability for an arbitrary dissipation coefficient $\nu>0$, whereas 
%Moreover, under Assumption (G2)(iii) we will also need to impose the condition $\nu >\frac{3 \tilde{K}_3^2}{2\lambda_1}$, meaning that the dissipation coefficient has to balance the strength of the \textit{multiplicative} part of the noise.

\begin{theorem}
\label{unique_thm}
Assume (G1) and (G2),   with the additional condition \eqref{condizione32} for the case (G2)(iii). 
Then there exists  a positive integer
\begin{equation*}
\bar N=\begin{cases}
\bar N(L_G, K_1, \nu, \|f\|_{V^*}) & \text{under (G2)(i)},
\\
\bar N(L_G, K_2, \tilde{K}_2, \nu, \lambda_1, \gamma, \|f\|_{V^*})  &	\ \text{under (G2)(ii)},
\\
\bar N(L_G, K_3, \tilde{K}_3, \nu, \lambda_1, \|f\|_{V^*}) &\text{under (G2)(iii)},
\end{cases}
\end{equation*}
 such that if (G3) holds for some  $M \ge \bar N$% $N\ge \bar N$
 , then $P$ possesses at most one ergodic invariant measure $\mu\in \mathcal{P}(H)$. 
\end{theorem}

\begin{remark}
\label{rem_ex}
Notice that Theorem \ref{unique_thm} yields also the existence of an invariant measure,
since the condition  \eqref{condizione32} on the viscosity coefficient is stronger 
than  the condition \eqref{visco1} assumed  to get  the existence result in Proposition \ref{prop_mis_inv}.
\end{remark}

We need some auxiliary result in order to prove Theorem \ref{unique_thm}.

In order to % to construct the generalized coupling and 
exploit Theorem \ref{GHMR}, the idea is to introduce a modification of the Navier-Stokes equation \eqref{NS_abs} such that: (i) the law of the solution to the new SPDE is absolutely continuous with respect to the law of the solution to the original one \eqref{NS_abs}; (ii)  for any pair of distinct initial conditions, there is a positive probability that solutions to these systems converge at time infinity, when evaluated  on a infinite sequence of evenly spaced times. 
%Points (i),(ii) are the ingredients for establishing a so-called \textit{asymptotic coupling}.

We start by introducing a modification of equation \eqref{NS_abs} such  that (i) holds and proving (see Lemma \ref{abs_cont_lem} below) that the law of the solution to this modified system \eqref{NS_abs_nud_stop} is absolutely continuous (actually equivalent) with respect to the law of the solution of the original system \eqref{NS_abs}, provided Assumption (G3) holds with $M \ge N$. Then, in Proposition \ref{prop_coupling} we prove that there exists $\bar N>0$ sufficiently large such that, if $N \ge \bar N$, % (and Assumption (G3) holds with $M \ge N$)
then there exists an infinite sequence of evenly spaced times such that for any pair of different initial conditions, the probability that the solutions to systems \eqref{NS_abs} and \eqref{NS_abs_nud_stop} converge at time infinity is strictly positive. Uniqueness of the invariant measure will then steam as a consequence of Lemma \ref{abs_cont_lem}, Proposition \ref{prop_coupling} and Theorem \ref{GHMR} imposing Assumption (G3) to hold with $M \ge \bar N$.

Let Assumption (G3) hold  for some $M \ge N$. We fix two initial conditions $u_0, v_0 \in H$ and we consider the Navier-Stokes equation \eqref{NS_abs} starting from $u_0$ and its nudged equation \eqref{NS_abs_nud} starting from $v_0$. 
%We recall that the parameters $N$, $\lambda$ that appear in \eqref{NS_abs_nud} have to be chosen.
We define the shift $h$ by 
\begin{equation}
\label{sigma}
h(t):= \lambda\ g(v(t))\ P_N(u(t)-v(t)), \quad t \ge 0,
\end{equation}
where $ \lambda P_N(u(t)-v(t))$ is the nudged term in equation  \eqref{NS_abs_nud}.
Notice that the definition of $h$ does make sense, thanks to assumption (G3) 
and the fact we required $M \ge N$. The shift $h$ belongs to the space $U$, where the noise lives.
Given $K>0$ we introduce the stopping time 
\begin{equation}\label{def_sigma}
\sigma_K:= \inf\left\{t \ge 0:  \int_0^t \|P_N(u(s)-v(s))\|^2_H {\rm d}s\ge K  \right\}.
\end{equation}
The constant $K$ will be chosen in a suitable way later on (see the proof of Proposition \ref{prop_coupling}). 
We set 
\begin{equation*}
\widetilde{W}(t):=W(t)+ \int_0^t h(s) \pmb{1}_{s \le \sigma_K}\, {\rm d}s.
\end{equation*}
The modified equation upon which we will build a generalized coupling is given by 
\begin{equation}
\label{NS_abs_nud_stop}
\begin{cases}
{\rm d}\tilde v(t) + \left[\nu A\tilde v(t)+B(\tilde v(t),\tilde v(t))\right]\,{\rm d}t
=G(\tilde v(t))\,{\rm d}\widetilde W(t)+f\,{\rm d}t
\\
\tilde v(0)=v_0
\end{cases}
\end{equation}
We will refer to \eqref{NS_abs_nud_stop} as the \textit{nudged stopped equation} corresponding to the Navier-Stokes equation \eqref{NS_abs}.
%In the next Lemma we prove that the law of the solution to the nudged stopped system \eqref{NS_abs_nud_stop} is absolutely continuous (actually equivalent) with respect to the law of the solution of the original system \eqref{NS_abs}, provided Assumption (G3) holds with $M \ge N$.

We denote by $\Psi_{u_0}$ and $\widetilde{\Psi}_{u_0,v_0}$ the measurable maps induced by solutions to \eqref{NS_abs} and \eqref{NS_abs_nud_stop}, respectively, that map an underlying probability space $(\Omega, \mathcal{F}, \mathbb{P})$ to $C([0, \infty);H)$. The law of solutions of \eqref{NS_abs} and \eqref{NS_abs_nud_stop} are given by $\mathbb{P} \Psi^{-1}_{u_0}$ and $\mathbb{P} (\widetilde{\Psi}_{u_0, v_0})^{-1}$ respectively.

\begin{lemma}
\label{abs_cont_lem}
Let assumptions  (G1) and (G2) %(either (i), (ii) or (iii))
be in force and let $N$ be the integer appearing in \eqref{sigma}. If Assumption (G3) holds with $M \ge N$, then for any $K,\lambda>0$, the laws of  solutions to \eqref{NS_abs} and \eqref{NS_abs_nud_stop} are equivalent (i.e. mutually absolutely continuous), that is $\mathbb{P}\Psi^{-1}_{v_0}\sim \mathbb{P}(\widetilde{\Psi}_{u_0,v_0})^{-1}$ as measures on $C([0, \infty);H)$.
\end{lemma}
%Proof
\begin{proof}

Bearing in mind \eqref{sigma}, we have  
\[\begin{split}
\int_0^\infty \|h(s)\|^2_{U}\pmb{1}_{s\le \sigma_K}{\rm d}s
%\int_0^\infty \lambda^2 \|g(v(s))P_N(u(s)-v(s))\|^2_{U}\pmb{1}_{s\le \tau_K}{\rm d}s
&\le 
 \lambda^2 \left(\sup_{x \in H}\|g(x)\|_{L(H,U)}^2\right)\int_0^\infty \|P_N(u(s)-v(s))\|^2_{H}\pmb{1}_{s\le \sigma_K}{\rm d}s
\\&\le
 \lambda^2 \left(\sup_{x \in H} \|g(x)\|_{L(H,U)}^2\right) K
\end{split}\]
which is finite thanks to \eqref{bound_g} in Assumption (G3). 
Therefore, the drift $h(s)\pmb{1}_{s\le \sigma_K}$  satisfies the Novikov condition 
\begin{equation*}
\mathbb{E} \left[ \exp\left(\frac 12 \int_0^\infty  \|h(s)\|^2_U\pmb{1}_{s\le \sigma_K}  \,{\rm d}s\right)\right]<\infty
\end{equation*}
and from the Girsanov Theorem we infer that there exists a probability measure $\mathbb{Q}$ on $%\Omega=
C([0, \infty);U)$
 such that under $\mathbb{Q}$,  $\widetilde{W}$ is a $U$-valued Wiener process on the time interval $[0, \infty)$.  It follows that the law of the solution to  the nudget stopped equation 
 \eqref{NS_abs_nud_stop} is equivalent on $C([0, \infty);H)$ to the law of the solution to the equation \eqref{NS_abs} with initial condition $v_0$, i.e. $\mathbb{P}\Psi^{-1}_{v_0}\sim \mathbb{P}(\widetilde{\Psi}_{u_0,v_0})^{-1}$ as measures on $C([0, \infty);H)$.
%\textcolor{cyan}{In particular, the law of the pair $(u(\cdot; u_0), \tilde v(\cdot, v_0))$ has marginals which are equivalent as measures on $C([0, \infty);H)$ to solutions to \eqref{NS_abs} starting respectively from $u_0$ and $v_0$.}
\end{proof}

\begin{remark}
\label{rem_eq}
On the set $\{\sigma_K=\infty\}$ we have that $v=\tilde v$, $\mathbb P$-a.s.,  
where $v$ is the solution of the nudged equation \eqref{NS_abs_nud}.
This steams from the uniqueness of the solution of equation \eqref{NS_abs_nud}.
\end{remark}

The crucial ingredient to prove the following result is given by the Foias-Prodi estimates of Proposition \ref{FP_cor}.
\begin{proposition}
\label{prop_coupling}
Assume (G1) and 
(G2),   with the additional condition {\color{purple}\eqref{condizione32}}
for the case (G2)(iii).
Let  $u$ be the solution of  the Navier-Stokes equation \eqref{NS_abs} with initial velocity  $u_0$ and 
$\tilde v$  the solution of the stopped nudget equation \eqref{NS_abs_nud_stop} 
with initial velocity $v_0$.

Then there exist a positive integer 
\begin{equation*}
\bar N=\begin{cases}
\bar N(L_G, K_1, \nu, \|f\|_{V^*}) & \text{under (G2)(i)},
\\
\bar N(L_G, K_2,\tilde{K}_2, \nu, \lambda_1, \gamma, \|f\|_{V^*})  &	\ \text{under (G2)(ii)},
\\
\bar N(L_G, K_3, \tilde{K}_3, \nu, \lambda_1, \|f\|_{V^*}) &\text{under (G2)(iii)},
\end{cases}
\end{equation*}
 and  a positive $\lambda=\lambda(\bar N, \nu)$ 
such that when $N \ge \bar N$ and $u_0, v_0 \in H$, 
one has
\begin{equation*}
\mathbb{P}\left(\lim_{n \rightarrow \infty}\|\tilde v(n)-u(n)\|_H=0 \right)>0.
\end{equation*}
%\textcolor{blue}{provided Assumption (G3) hold with $M \ge N$.} \textcolor{red}{se cambi def stoppimg time questa NON serve: CHECk e cambiare nel seguito eventiualmenet}
\end{proposition}
%PROOF
\begin{proof}
For any  $n \in \mathbb{N}$ we introduce the events %\textcolor{red}{Se cambi def stopping time $\tau_K$ togli $g$ sotto dappertutto.}
\begin{equation}
B_n:= \left\{ \|v(n)-u(n)\|^2_H +\int_n^{n+1}\|P_N(v(s)-u(s)\|^2_H\, {\rm d}s > \frac{1}{n^2}\right\},
\end{equation}
and, for $R$ and $m>0$ to be chosen later on,
\begin{equation}
\label{E_R_m}
E_{R,m}:=\left\{\int_0^{m}\|P_N(v(s)-u(s)\|^2_H\,{\rm d}s  > R\right\}.
\end{equation}
%Notice that, since we required Assumption (G3) to hold with $M \ge N$, the above sets are well defined.
We set 
\begin{equation*}
B:= \bigcap_{m=1}^{\infty} \bigcup_{n=m}^{\infty} B_n.
\end{equation*}
Next we fix any suitably large value of $\bar N$ so that either Proposition \ref{tau_exp} or \ref{tau_pol_ii} or \ref{tau_pol_iii} (according to which assumption of the operator $G$ we consider: either (G2)(i) or (G2)(ii) or (G2)(iii)) and Corollary \ref{cor_tau_R_beta} hold. We fix $N\ge\bar N$ and set $\lambda=\frac{\nu\lambda_N}2$ in the nudged equation \eqref{NS_abs_nud}. 
We consider the stopping time $\tau_{R, \beta}$ defined in \eqref{tau_R_beta} and
 write 
\begin{equation*}
\mathbb{P}(B)=\mathbb{P}\left(B \cap \{\tau_{R, \beta}=\infty\} \right)+ \mathbb{P}\left(B \cap \{\tau_{R, \beta}<\infty\} \right).
\end{equation*}
Thanks to the Borel-Cantelli lemma we have that $\mathbb{P}\left(B \cap \{\tau_{R, \beta}=\infty\} \right)=0$ for any $R, \beta>0$. 
In fact, thanks to the Chebychev inequality, the Fubini theorem and Corollary \ref{cor_tau_R_beta}, for any $n \in \mathbb{N}$
\begin{align*}
\mathbb{P}\left(B_n \cap \{\tau_{R, \beta}=\infty\} \right)
&\le n^2 \mathbb{E} \left[\pmb{1}_{(\tau_{R, \beta}=+\infty)} \left(\|v(n)-u(n)\|^2_H+ \int_n^{n+1}\|v(s)-u(s)\|^2_H\, {\rm d}s\right)\right]
\\
&\lesssim_{N,\nu} e^{R+ \beta}\|u_0-v_0\|^2_H n^2 e^{-\frac{\nu \lambda_N}{4}n},
\end{align*}
so that $\sum_{n=1}^{\infty}\mathbb{P}(B_n \cap \{\tau_{R, \beta}=\infty\})< \infty$. Hence 
$\mathbb{P}\left(B \cap \{\tau_{R, \beta}=\infty\} \right)=0$. 

Thus, along with the estimates \eqref{tau_est_i}, \eqref{tau_est_ii} and \eqref{tau_est_iii} from Propositions \ref{tau_exp}, \ref{tau_pol_ii} and \ref{tau_pol_iii}, respectively, we can select suitable values of $\beta$ such that, for any $R>0$, we have
\[
\mathbb{P}(B)=\mathbb{P}(B \cap \{\tau_{R, \beta}< \infty\}) \le \mathbb{P}(\tau_{R, \beta}< \infty)
\le 
\begin{cases}
e^{-CR} & \text{under (G2)(i)} %, \quad \text{with $C=C(\lambda_1, \nu, K_1)$},
\\[2mm]
\dfrac{C}{R^p} & \text{ for any } p>0, 	\ \text{under (G2)(ii)} %, \quad \text{with $C=C(\|u_0\|_H, \lambda_1, p, \nu,K_2, \tilde{K}_2, \gamma, \|f\|_{V^*})$},
\\[2mm]
\dfrac{C}{R^p} &  \text{ for  }p\in(0,\frac{\nu \lambda_1}{2 \tilde{K}_3^2}- \frac 34), 	\ \text{under (G2)(iii)} %,\quad  \text{with $C=C(\|u_0\|_H, \lambda_1, p, \nu,K_3, \tilde{K}_3, \|f\|_{V^*})$},
\end{cases}
\]
where the above constants $C$ do not depend on $R$.
%\begin{footnote}{\textcolor{red}{Stiamo usando Proposition \ref{cor_tau_R_beta_bis} con $q=2$: va aggiunta nel thm unicit\'a inv measure l'assunzione $\nu>\frac{7\tilde{K_1}^2}{2\lambda_1}$.}}\end{footnote}
By choosing  %$R^*=R^*(\lambda_1, \nu, K_1, \tilde{K_1}, \|u_0\|_H^8)$
$R^*$
 sufficiently large, we have that $\mathbb{P}(B)$
 is close to $0$, hence $\mathbb{P}(B^c)$ is close to $1$.
  Hence, from the continuity from below, we can thus find $m^*>0$ sufficiently large so that 
\begin{equation*}
\mathbb{P}\left( \bigcap_{n=m^*}^{\infty}B_n^c \right)>\frac34. 
\end{equation*}
For this fixed value of $m^*$ we now consider the set $E_{R^*, m^*}$ introduced in \eqref{E_R_m}. %Assumption (G3), 
The Chebychev inequality, \eqref{sup_q_est} and \eqref{est_v}  with $q=2$ yield
\begin{align*}
\mathbb{P}(E_{R^*, m^*})& \le \frac{\mathbb{E}\left[ \int_0^{m^*}\|P_N(v(s)-u(s)\|^2_H\,{\rm d}s\right]}{R^*}
\le \frac{C m^* }{R^*}, 
\end{align*}
for some constant $C$ dependent on the initial data $\|u_0\|^2_H, \|v_0\|^2_H$ and the parameters that appears in the statement of Lemmata \ref{lem_A2} and \ref{lem_v}.
% on the initial data =C(\nu, \lambda_1, K_1, \tilde{K_1},\textcolor{blue}{\|u_0\|^2_H, \|v_0\|^2_H})$.
It then follows, upon taking $R^*$ possibly larger,  that $\mathbb{P}(E^c_{R^*, m^*})>\frac 34$, hence
\begin{footnote}
{Here we use the inequality $\mathbb{P}(A \cap B)\ge \mathbb{P}(A)+\mathbb{P}(B)-1$.}
\end{footnote}
\begin{equation}
\label{sti_uniq_2}
\mathbb{P}\left(E_{R^*,m^*}^c \cap \bigcap_{n=m^*}^{\infty}B_n^c \right) >\frac 12.
\end{equation}
At this point we notice that, on the set $E_{R^*,m^*}^c \cap \bigcap_{n=m^*}^{\infty}B_n^c$, 
by splitting the integral as the sum of the integrals over the time intervals $[0,m^*]$, $[m^*, m^*+1]$, 
$[m^*+1, m^*+2]$ and so on,  we have 
\begin{equation*}
 \int_0^\infty\|P_N(v(s)-u(s)\|^2_H\,{\rm d}s
 \le 
 R^* + \sum_{n=m^*}^{\infty} \frac{1}{n^2}  < \infty.
 \end{equation*}
 We now choose 
 \[
 K=R^* + \sum_{n=m^*}^{\infty} \frac{1}{n^2} 
 \]
 as a parameter defining the stopping time $\sigma_K$ defined in \eqref{def_sigma}.
 Notice the two inclusions
 \[
 E_{R^*,m^*}^c \cap \bigcap_{n=m^*}^{\infty}B_n^c
\  \subseteq\  \{\sigma_K=\infty\}
 \]
 and for any $m^*$
 \[
  \bigcap_{n=m^*}^{\infty}B_n^c\subseteq 
\Big \{\lim_{n\to\infty} \|u(n)-v(n)\|^2_H =0\Big\}.
 \]
Thus it follows
\begin{multline}
%\label{sti_uniq_1}
\mathbb{P} \left(\lim_{n \rightarrow \infty}\|\tilde v(n)-u(n)\|^2_H=0 \right)
 \ge \mathbb{P}\left(\lim_{n \rightarrow \infty}\|\tilde v(n)-u(n)\|^2_H=0  \ \cap \{\sigma_K=+ \infty\}\right)
\\
= \mathbb{P}\left(\lim_{n \rightarrow \infty}\|v(n)-u(n)\|^2_H=0  \ \cap \{\sigma_K=+ \infty\}\right)
\ge \mathbb{P} \left( E_{R^*,m^*}^c \cap \bigcap_{n=m^*}^{\infty}B_n^c \right)
\end{multline}
where the  equality in the above relation steams from the fact that $v=\tilde v$ on $\{\sigma_K=+ \infty\}$ (see Remark \ref{rem_eq}).
 
 %Thus it follows that $E_{R^*,m^*}^C \cap \bigcap_{n=m^*}^{\infty}B_n^C \subseteq \{\tau_K=\infty\} \cap \{\lim_{n\rightarrow \infty}\|u(n)-v(n)\|_H=0\}$, i.e. $v(t)=\tilde v(t)$ for all $t \ge 0$ on $E_{R^*,m^*}^C \cap \bigcap_{n=m^*}^{\infty}B_n^C$.

Therefore for the previous choice of the parameters $R^*$, $m^*$ and $K$, from \eqref{sti_uniq_2}  we obtain 
\begin{equation*}
\mathbb{P}\left(\lim_{n \rightarrow \infty}\|\tilde v(n)-u(n)\|^2_H=0 \right)
\ge \mathbb{P} \left( E_{R^*,m^*}^c \cap \bigcap_{n=m^*}^{\infty}B_n^c \right)
>\frac 12
\end{equation*}
and this concludes the proof.
\end{proof}

We are ready to prove Theorem \ref{unique_thm}.
\begin{proof}[Proof of Theorem \ref{unique_thm}]
The uniqueness of the invariant measure is a consequence of Lemma \ref{abs_cont_lem} and Proposition \ref{prop_coupling} thanks to which we verify the assumptions of Theorem \ref{GHMR}. The proof is as follows.
%We aim to show that $P_1$, the Markov kernel associated to \eqref{NS_abs} via \eqref{P_t_kernel} at time $t=1$, 
%\begin{footnote}{\textcolor{blue}{Consideriamo $P_t$ al tempo $t=1$ perch\'e in questo modo ci muoviamo a steps di $1$ e quindi siamo nel framework $H^{\mathbb{N}}$ del risultato di coupling}}\end{footnote} admits at most one invariant measure. 
For any $u_0, v_0 \in H$ we consider the measure $\xi_{u_0,v_0}$ on $H^{\mathbb{N}} \times H^{\mathbb{N}}$ given by the law of $(u(n), \tilde v(n))_{n \in \mathbb{N}}$, where $u$ and $\tilde v$ solve equations  \eqref{NS_abs} and \eqref{NS_abs_nud_stop}, respectively, with corresponding initial data $u_0, v_0$. Thanks to Lemma \ref{abs_cont_lem} we have that, provided Assumption (G3) holds with $M \ge N$, $\pi_2(\xi_{u_0,v_0}) \sim \mathbb{P}_{v_0}$. We therefore have that $\xi_{u_0, v_0} \in \widetilde{\mathcal{C}}\left(\mathbb{P}_{u_0}, \mathbb{P}_{v_0}\right)$. From the definition of $\xi_{u_0, v_0}$ and Proposition \ref{prop_coupling} we have 
\begin{equation*}
\xi_{u_0, v_0}(D)%= \mathbb{P}\left((u(n),\tilde v(n))_{n \in \mathbb{N}} \in D \right)
=\mathbb{P}\left( \lim_{n \rightarrow \infty}\|\tilde v(n)-u(n)\|_H=0 \right) >0,
\end{equation*}
for the suitable choice of parameters $\lambda, N, K$ (that appears in the equation for $\tilde v$) made in Proposition \ref{prop_coupling}; in particular, $N\ge \bar N$, with $\bar N$ as in Proposition \ref{prop_coupling}. Since the test functions $\mathcal{G}$ defined in \eqref{G} determine measures on $(H, \|\cdot\|_H)$, thanks to Theorem \ref{GHMR} we conclude that there exists at most one invariant measure for $P$ in $\mathcal{P}(H)$, provided Assumption (G3) holds with $M \ge\bar N$. 
%This in turns means that any two invariant measures for the entire semigroup $\{P_t\}_{t\ge 0}$ must coincide since they are both invariant measures for $P_1$.
\end{proof}

\subsection{Asymptotic stability}
\label{asy_sta_sec}
Let us now come to the issue of asymptotic stability of the invariant measure.

 \begin{theorem}
\label{asy_sta}
Assume (G1) and 
(G2),   with the additional condition
\begin{equation}\label{condizione72}
 \nu > \frac{11\tilde{K}^2_3}{2\lambda_1}
 \end{equation}
for the case (G2)(iii). 
Then there exists  a positive integer 
\begin{equation*}
\bar N=\begin{cases}
\bar N(L_G, K_1, \nu, \|f\|_{V^*}) & \text{under (G2)(i)},
\\
\bar N(L_G, K_2, \tilde{K}_2, \nu, \lambda_1, \gamma, \|f\|_{V^*})  &	\ \text{under (G2)(ii)},
\\
\bar N(L_G, K_3, \tilde{K}_3, \nu, \lambda_1, \|f\|_{V^*}) &\text{under (G2)(iii)},
\end{cases}
\end{equation*}
 such that if (G3) holds with $M \ge \bar N$% $N\ge \bar N$
 , then the transition semigroup $P$ associated to equation \eqref{NS_abs} possesses at most one ergodic invariant measure $\mu$ on $H$ and
\begin{equation*}
\lim_{t \rightarrow \infty} \|P_t^*\delta_{u_0}-\mu\|_*= 0 \quad \forall \ u_0 \in H.
\end{equation*}
\end{theorem}
%PROOF
\begin{proof}
The proof relies on Theorem \ref{KS} and it is based on a stochastic control argument similar to the one developed in Section \ref{main_result_sec}. However here we have to drop the localization term $\pmb{1}_{\sigma_K>t}$ in order to satisfy the assumptions of Theorem \ref{KS}. This is not an issue since, exploiting the Foias-Prodi estimates, we can show that the law of the pair $(u, v)$, with $u$ the solution to the Navier-Stokes equation
\eqref{NS_abs} and $v$ the solution to its nudged equation \eqref{NS_abs_nud}, is a generalized coupling that satisfies the assumptions of Theorem \ref{KS}.
The consequences of not using the localization term are seen only when working under assumption (G2)(iii) where the condition on viscosity becomes even stronger (see Remark \ref{rem_asy_sta} for further comments).

For the sake of exposition we divide the proof in three steps.
%\textcolor{cyan}{Ho circa riscritto la proof come prima: seguiva gli steps della Section 5.3, e mi sembrava questo aiutasse a seguire la dimostrazione.}
\begin{itemize}
\item [(i)]
Let Assumption (G3) hold with $M \ge N$. Set 
\begin{equation*}
h(t):= \lambda g(v(t))(P_N(u(t)-v(t)), \quad t \ge 0,
\end{equation*}
with $\lambda=\frac{\nu \lambda_N}{2}$, and define 
\begin{equation*}
\widetilde{W}(t):=W(t)+ \int_0^t h(s)\, {\rm d}s, \quad t\ge 0.
\end{equation*}
The equation \eqref{NS_abs_nud} for $v$ can be written in the form
\begin{equation}
{\rm d}v(t) + \left[\nu Av(t)+B(v(t),v(t))\right]\,{\rm d}t=G(v(t))\,{\rm d}\widetilde{W}(t)+ f \,{\rm d}t.
\end{equation}
Given any positive constant $c>0$, by the Chebychev inequality we infer
\begin{align}
\label{media-quadrato}
\mathbb{P}\left( \int_0^\infty\|h(s)\|_U^2\, {\rm d}s>c\right )
&\le \frac 1c \mathbb{E} \left[ \int_0^\infty\|h(s)\|_U^2 \, {\rm d}s\right]
\\ \notag
&\le \frac{\lambda^2}{c}\left( \sup_{x\in H}\|g(x)\|^2_{L(U,H)}\right)\mathbb{E}\int_0^\infty\|u(s)-v(s)\|_H^2\, {\rm d}s.
\end{align}
The term depending on $g$ is bounded thanks to assumption \eqref{bound_g}. 
Moreover, Theorem \ref{FP_cor} yields  
\begin{equation}
\label{f}
\mathbb{E}\left[ \|u(s)-v(s)\|^2_H\right]\le f(s):=
\begin{cases}
C e^{-\delta s}, & \text{under (G2)(i)}
\\[2mm]
\dfrac{C}{s^p}, \quad  \forall \ p>0, & \text{under (G2)(ii)}
\\[2mm]
\dfrac{C}{ s^p},  \quad \forall  \ p \in ( 0,\frac{\nu \lambda_1}{4 \tilde{K}_3^2}-\frac 38), & \text{under (G2)(iii)}
\end{cases}
\end{equation}
with $C$ positive constants depending on the parameters of the equations and the initial data but independent of $s$. Thus 
\[
\int_0^\infty  \mathbb{E}  \left[ \|u(s)-v(s)\|^2_H \right]\, {\rm d}s 
\le 
\int_0 ^\infty f(s) {\rm d}s.
\]
We consider the latter integral;  under (G2)(i) it is a finite number, and the same holds under (G2)(ii) or (G2)(iii) by choosing  $p>$1. We notice that under  (G2)(iii) it is necessary that $1<\frac{\nu \lambda_1}{4 \tilde{K}_3^2}-\frac 38$, which explains \eqref{condizione72}.

Thus, by letting $c$ go to infinity in \eqref{media-quadrato} we infer
\begin{equation}
%\label{P_bound}
\mathbb{P}\left(\int_0^\infty \|h(s)\|^2_U\, {\rm d}s< \infty \right)=1.
\end{equation}
By the Girsanov Theorem the law of $\widetilde{W}$ is absolutely continuous w.r.t. the law of $W$. In turns, the law of the solution $v$ to the nudged equation \eqref{NS_abs_nud} is absolutely continuous w.r.t. the law of the solution $u$ to equation \eqref{NS_abs} with initial datum $v_0$, as measures on $C([0, \infty);H)$. %\textcolor{cyan}{Note that the law of the first component w.r.t. this generalized coupling just equals $\mathbb{P}_{u_0}$.}

\item [(ii)] We check condition \eqref{hyp_KS} of Theorem \ref{KS}.

Fix and choose $\varepsilon>0$. From the Foias-Prodi estimates in Theorem \ref{FP_cor}, provided $N \ge \bar N$ (where $\bar N$ is as in Theorem \ref{FP_cor}), we infer 
\begin{equation}
\mathbb{P}\left(\|u(n)-v(n)\|^2_H >\varepsilon \right) \le \frac{1}{\varepsilon}\mathbb{E}\left[ \|u(t)-v(t)\|^2_H \right] \le \frac{f(n)}{\varepsilon}
\end{equation}
with $f$ as in \eqref{f}. Since $f(n) \rightarrow 0$ as $n \rightarrow \infty$, it follows that 
\begin{equation}
\lim_{n \rightarrow \infty}\mathbb{P}\left(\|u(n)-v(n)\|^2_H \le\varepsilon \right) =1.
\end{equation}

\item[(iii)] Steps (i) and (ii) lay the ground to apply Theorem \ref{KS}.
First we observe that the semigroup $P$ is Feller, as already proved in Proposition \ref{prop_mis_inv}.
For any $u_0, v_0 \in H$ we consider the measure $\xi_{u_0,v_0}$ on $H^{\mathbb{N}} \times H^{\mathbb{N}}$ given by the law of the associated random vector $(u(n), v(n))_{n \in \mathbb{N}}$, where $u$ and $v$ solve \eqref{NS_abs} and \eqref{NS_abs_nud}, respectively, with corresponding initial data $u_0, v_0$. We have that $\pi_1(\xi_{u_0,v_0})= \mathbb{P}_{u_0}$. Moreover, from Step (i), we have that $\pi_2(\xi_{u_0,v_0})\sim \mathbb{P}_{v_0}$, provided Assumption (G3) holds with $M \ge N$. Thus $\xi_{u_0,v_0} \in \widetilde C(\mathbb{P}_{u_0}, \mathbb{P}_{v_0})$.
%with, in particular, $\pi_1(\xi_{u_0,v_0})\sim \mathbb{P}_{u_0}$. 
From the definition of $\xi_{u_0, v_0}$ and Step (ii) we have, for any $\varepsilon>0$,
\begin{equation*}
\lim_{n \rightarrow \infty}\xi_{u_0, v_0}(D_\varepsilon)= \lim_{n \rightarrow \infty}\mathbb{P}\left((u(n),v(n))_{n \in \mathbb{N}} \in D_\varepsilon  \right)=\lim_{n \rightarrow \infty}\mathbb{P}\left(\|v(n)-u(n)\|_H \le \varepsilon \right) =1,
\end{equation*}
imposing $N\ge \bar N$, with $\bar N$ as in Theorem \ref{FP_cor}. Since the assumptions of Theorem \ref{KS} are verified, we conclude that there exists $\bar N$ sufficiently large such that, provided Assumption (G3) holds with $M \ge\bar N$, there exists
 at most one invariant measure for $P$ which is asymptotically stable.
 \end{itemize}
 \vspace{-5mm}
\end{proof}

\begin{remark}
\label{rem_asy_sta}
%Working under (G2)(i)-(ii), Theorem \ref{asy_sta} would give us at once existence, uniqueness and asymptotic stability of the invariant measure. Instead under (G2)(iii) the results of existence and uniqueness of the invariant measure would have been obtained under stronger restrictions on the viscosity than those obtained in Proposition \ref{prop_mis_inv} and Theorem \ref{unique_thm}.Put another way, 
The introduction of the localization term $\pmb{1}_{s \le \sigma_k}$ is entirely superfluous working under (G2)(i)-(ii), while it allows the condition on the dissipation coefficient to be weakened by working under (G2)(iii). We observe that in \cite{KS} the authors, having to deal with an additive noise, emphasize that the localization term is entirely superfluous to show the uniqueness and asymptotic stability of the invariant measure: the reason is roughly speaking that in the proof they exploit \textit{pathwise} Foias-Prodi estimates  with an exponential decay. Our condition (G2)(i) most closely resembles the case considered in \cite{KS}. \end{remark}

\section{Final remarks}
\label{final_rem_sec}

We have shown how  the asymptotic generalized coupling techniques from \cite{GHMR17} and \cite{KS} can be successfully adapted to prove the uniqueness and the asymptotic stability of the invariant measure 
for the stochastic Navier-Stokes equations in the presence of multiplicative noise in an effectively elliptic setting. 

The key tool for proving these results are the Foias-Prodi estimates in expected value. These show different decay in time depending on the noise assumptions: exponential in the bounded noise case (compare the result with \cite{Oda2008}), polynomial for any exponent $p>0$ in the sublinear growth noise case, polynomial for any exponent $0<p<\frac{\nu \lambda_1}{4 \tilde{K}_3^2}-\frac 38$ in the linear growth noise case. Working under (G2)(i)-(ii) the Foias-Prodi estimates have sufficiently nice behavior to prove the results of uniqueness and asymptotic stability of the invariant measure at once. In these cases there is no need to introduce the localization term $\pmb{1}_{s \le \sigma_k}$ i.e. it is sufficient to use the techniques of \cite{KS} and the result follows from Theorem \ref{asy_sta}. Instead, the Foias-Prodi estimates that we obtain by working under (G2)(iii) impose the condition $0<p<\frac{\nu \lambda_1}{4 \tilde{K}_3^2}-\frac 38$ on the admissible parameters that give the polynomial decay. This fact has consequences for the conditions to be imposed on the viscosity coefficient in order to have uniqueness ($\nu>\frac{3\tilde K_3^2}{2\lambda_1}$) and asymptotic stability ($\nu>\frac{11\tilde K_3^2}{2\lambda_1}$) of the invariant measure: the localization term we can introduce to use the results of \cite{GHMR17} allows for the weaker condition. We observe that in the case of a noise with a linear growth we should not be surprised that a condition on viscosity appears: see, for example, \cite{Fla_Gat} where a condition  appeared just for the existence of the invariant measure (compare with Proposition \ref{prop_mis_inv} where the condition  $\nu>\frac{\tilde K_3^2}{2\lambda_1}$ appears). 

We conclude by pointing out that the problem we addressed here was inspired by Remark 3.7 in \cite{Oda2008} in which Odasso predicts (without proving it) polynomial mixing 
when the  covariance of the noise has  linear or sublinear growth, that is when we assume (G2)(ii) or (G2)(iii).
Our aim has been to deal with these assumptions and we obtained the uniqueness of the invariant measure and the  convergence  to it for large time. No quantitative mixing results are available so far, as in \cite{GHMR17} and \cite{KS} ,  but this  is under investigation. Finally, with a bounded multiplicative noise our technique  is  simpler than that of 
\cite{Oda2008}.

We expect that the different decays in time in our Foias-Prodi estimates, depending on the  three noise assumptions,  should lead to demonstrating different types of quantitative mixing (exponential/polynomial).

%\begin{itemize}\item lavorando sotto (G2)(i)-(ii) si sarebbe potuto applicare direttamente risultato KS, le cose cambiano sotto (G2(iii): qua introdurre lo stopping time ci permette di indebolire le Hyp sulla dissipazione.\begin{remark}Dire che usando le FP estimates potremmo fare a meno dello stopping time: non cambia assolutamente nulla nel caso di noise bounded e sublineare, anzi la proof si semplifica; per il caso bounded noise, senza lo stopping time vanno per\'o rafforzate le ipotesi sulla viscosit\'a. \end{remark}\item Le tecniche di generalized coupling funzionano anche per rumori moltiplicativi (come caso particolare si ottengono i risultati per noise additivo).\item Abbiamo considerato diverse assunzioni sul noise: meno restrittive che in Odasso.\item Ci aspettiamo che il tipo di decadimento delle FP estimates conduca a diversi tipi di mixing quantitativo. Questo \'e oggetto di studio al momento.\item Dire che avremmo potuto lavorare sempre sotto (G3)(iii) ma abbiamo preferito separare i casi perch\'e pensiamo che le diverse FP portino a mixing quantitativi diversi. Nel mixing qualitativo che abbiamo ottenuto questa cosa si perde.\end{itemize}Potremmo applicare KS e ottenere il risultato forte in un colpo solo nel noise bounded e sublinar. Le cose sono pi\'u complicate nel noise a crescita lineare perch\'e qua salta fuori una condizione sulla viscosit\'a (gi\'a evidenziata in Flandoli per esistenza inv.meas.).

%%%%%%%%%%%%%%%%%%%%%%%%%%%%%%%%%%%%%%

\appendix
\section{A priori estimates}\label{app_A}
In this Appendix we collect some apriori estimates on the solution to the Navier-Stokes equation \eqref{NS_abs} and  its nudged equation \eqref{NS_abs_nud}. We recall that we assume $u_0 \in H$, $\nu>0$ and $f\in V^*$.

\subsection{Moment estimates}
The following two lemmata collect some a priori  estimates and moments bounds on the solution to system \eqref{NS_abs}, according to the different Assumptions (G2)(i), (ii) or (iii).
%LEMMA
\begin{lemma}
\label{lem_A1}
Assume (G1)-(G2) with the additional condition 
\begin{equation}\tag{\ref{visco1}}
\nu>\frac{\tilde{K}_3^2}{2\lambda_1}
\end{equation}
for the case (G2)(iii). Then there exist positive constants 
$a$ an $b$ such that the strong solution to 
the Navier-Stokes equation  \eqref{NS_abs} satisfies 
\begin{equation}\label{stima-media-energia}
\mathbb{E}\left[\|u(t)\|^2_H\right]+ a\int_0^t \mathbb{E}\left[\|u(s)\|^2_V\right]\, {\rm d}s \le \|u_0\|^2_H+bt, \qquad t >0, 
\end{equation}
with 
\begin{equation}
\label{a}
a=
\begin{cases}
\nu, & \text{under Assumption (G2)(i) or (G2)(ii)}
\\
\nu-\frac {\tilde K_3^2}{2\lambda_1}, & \text{under Assumption (G2)(iii)}
\end{cases}
\end{equation}
and 
\begin{equation}
\label{b}
b=
\begin{cases}
K_1^2 + \frac 1 \nu\|f\|^2_{V^*}, & \text{under Assumption (G2)(i)}
\\
b_1+ b_2\|f\|^2_{V^*},& \textit{under Assumption (G2)(ii)}
\\
b_3 K_3^2 +b_4 \|f\|^2_{V^*}, & \text{under Assumption (G2)(iii)}
\end{cases}
\end{equation}
where $b_1=b_1( \nu,  \lambda_1, K_2, \tilde{K}_2,\gamma)$, $b_2=b_2(\nu, \lambda_1, \tilde{K}_2)$, $b_3=b_3(\nu, \lambda_1, K_3,\tilde{K}_3)$
and $b_4=b_4(\nu, \lambda_1, \tilde{K}_3)$ are positive constants. 
\end{lemma}
%PROOF
\begin{proof}
Let $u$ be the solution to \eqref{NS_abs}.
We apply the It\^o formula to  $\|u(t)\|^2_H$. 
Exploiting \eqref{B=0} we infer, $\mathbb{P}$-a.s., for any $t\ge0$,
\begin{multline}
\label{sti1}
\|u(t)\|^2_H + 2 \nu\int_0^t \|u(s)\|^2_V\, {\rm d}s  = \|u_0\|_H^2+ \int_0^t\|G(u(s))\|^2_{L_{HS}(U,H)}{\rm d}s 
\\
 +2 \int_0^t\langle u(s), G(u(s))\, {\rm d}W(s)\rangle + 2 \int_0^t \langle u(s), f\rangle\, {\rm d}s.
\end{multline}
The Young inequality yields, for  arbitrary $\varepsilon>0$
\begin{equation}
\label{G(u)_est}
\|G(u)\|^2_{L_{HS}(U,H)}
\le 
\begin{cases}
K_1^2 & \text{under  (G2)(i)},
\\
2K_2^2 +2 \tilde{K}_2^2 \|u\|^{2\gamma}_H 
\le 2K_2^2 + C(\varepsilon, \gamma) + \varepsilon \frac{ \tilde{K}_2^2}{\lambda_1}\|u\|^2_V & \text{under (G2)(ii)},
\\
(1+\frac 1 \varepsilon)K_3^2 + (1+ \varepsilon)\tilde{K}_3^2 \|u\|^2_H 
         \le (1+\frac 1 \varepsilon)K_3^2 + (1+ \varepsilon) \frac{\tilde{K}_3^2}{\lambda_1}\|u\|^2_V& \text{under (G2)(iii)},
\end{cases}
\end{equation}
and,
for arbitrary $\eta>0$
%\begin{equation}\label{sti_f}2\langle u, f\rangle \le 2\|u\|_V\|f\|_{V^*} \le \nu \|u\|_V^2+ \frac{1}{\nu} \|f\|^2_{V^*},\end{equation}
\begin{equation}
\label{sti_f}
 2\langle u, f\rangle \le 2\|u\|_V\|f\|_{V^*} \le \eta \nu \|u\|_V^2+ \frac{1}{\eta\nu} \|f\|^2_{V^*}.
 \end{equation}
 From \eqref{sti1}, \eqref{G(u)_est} and \eqref{sti_f} we thus obtain, $\mathbb{P}$-a.s., for all $t \ge0$,
\begin{equation}
\label{sti_gen}
\|u(t)\|^2_H + a\int_0^t \|u(s)\|_V^2\, {\rm d}s \le \|u_0\|_H^2 + bt + M(t),
\end{equation}
where the (local)  martingale term is 
\begin{equation}
\label{M}
M(t):= 2 \int_0^t \langle u(s), G(u(s))\, {\rm d}W(s)\rangle,
\end{equation}
and $a, b$ are as in \eqref{a}, \eqref{b} respectively. 
More precisely, the expression of $a$ in case (G2)(i) steams from choosing $\eta=1$, whereas  the expression of $a$ in case (G2)(ii) steams from choosing $\varepsilon$ and $\eta$ small enough so that $(2-\eta)\nu-\varepsilon  \frac{\tilde{K}_2^2}{\lambda_1}\ge \nu$. 
In case (G2)(iii), the coefficient in front of $\int_0^t \|u(s)\|_V^2\, {\rm d}s$ is 
$(2-\eta)\nu -   (1+ \varepsilon) \frac{\tilde{K}_3^2}{\lambda_1}$; 
 thanks to assumption \eqref{visco1} we  can find   $\varepsilon$ and $\eta$  small enough so that 
$(2-\eta)\nu -   (1+ \varepsilon) \frac{\tilde{K}_3^2}{\lambda_1}= \nu-\frac{\tilde{K}_3^2}{2\lambda_1} $
so we conclude the estimate. 
%{\color{cyan}Nel caso (G2)(iii) sto considerando $\varepsilon=1$ and $\eta = \frac 12$:{\color{purple}NO, SERVONO $\varepsilon$ E $\eta$  PICCOLI: $\underbrace{\qquad\nu-\frac{\tilde{K}_3^2}{2\lambda_1}\qquad}_{>0, \text{ pu\`o essere molto piccolo}}=\eta \nu + \varepsilon \frac{\tilde{K}_3^2}{\lambda_1} $} per queste scelte non ottengo $b_3=\frac 32$ e $b_4=\frac{1}{\nu}$?}
\\
Let us now observe that the stochastic integral is indeed a martingale, in fact we can estimate its quadratic variation as
\begin{equation*}
%\mathbb{E} \left[\int_0^t|\langle u(s), G(u(s))\, {\rm d}W(s)\rangle|^2 \right]
[M](t) \le 4\int_0^t\|u(s)\|^2_H \|G(u(s))\|^2_{L_{HS}(U,H)}\,{\rm d}s,
\end{equation*}
which is bounded thanks to Assumption (G2) and \eqref{q_mom_sol}.
%under either (i), (ii) or (iii) in Assumption (G2), thanks to \cite[Appendix A]{Fla_Gat} (there the authors works under Assumption (G2) (iii) which is the less restrictive of the three we take into account.) 
Therefore by taking the expected values on both sides of \eqref{sti_gen} we get the thesis. 
\end{proof}

\begin{lemma}
\label{lem_A2}
Assume (G1).
\begin{enumerate}[label=$(\roman{*})$]
\item If (G2)(i) holds, then for any $q\ge 2$ the solution to equation  \eqref{NS_abs} satisfies
%\begin{equation}\mathbb{E}\left[\|u(t)\|^{q}_H \right]\le C\left(1+ \|u_0\|^q_He^{-\widetilde Ct}\right)\qquad \forall t>0,\end{equation}
\begin{equation}
\label{sup_q_est}
\mathbb{E}\left[\|u(t)\|^{q}_H \right]\le \underline C+ \|u_0\|^q_He^{-\bar Ct}
\qquad \forall t>0,
\end{equation}
 where $\bar C=C(q,\nu,\lambda_1)>0$ and $ \underline C=C(q,\nu,\lambda_1, K_1, \|f\|_{V^*})>0$.
\item If (G2)(ii) holds, then for any $q\ge 2$ the solution to equation  \eqref{NS_abs} satisfies \eqref{sup_q_est} with $\bar C=\bar C(q,\nu,\lambda_1)>0$ and $\underline C=\underline C (q,\nu,\lambda_1, K_2, \tilde{K}_2, \gamma, \|f\|_{V^*},\gamma)>0$.
 \item If  (G2)(iii) and \eqref{visco1} hold, 
 then for any $q\in [2,  1+\frac{2\nu \lambda_1}{\tilde{K}_3^2})$ the solution to equation  \eqref{NS_abs} satisfies \eqref{sup_q_est} with $\bar C=\bar C(q,\nu,\lambda_1, \tilde{K_3})>0$ and $\underline C=\underline C (q,\nu,\lambda_1, K_3, \tilde{K_3}, \|f\|_{V^*})>0$.
\end{enumerate}
%\begin{equation}\label{E_sup_q}\mathbb{E}\left[ \sup_{0\le s\le t}\|u(s)\|^q\right] \le C,\end{equation}where $C$ is a positive constant depending on $\|u_0\|_H, t, K_1, \tilde{K}_1$ and $q$. 
\end{lemma}
%PROOF
\begin{proof}
We start dealing with the case $q=2$.
Using the Poincar\'e inequality \eqref{lambda_1} in the estimate \eqref{stima-media-energia} we obtain 
\[
\mathbb{E}\left[\|u(t)\|^2_H\right]+ \frac a {\lambda_1}\int_0^t \mathbb{E}\left[\|u(s)\|^2_H \right]\, {\rm d}s \le \|u_0\|^2_H+bt
\]
and  we conclude  thanks to the Gronwall lemma.

Now we observe that statement (i) can be proved as statement (ii) by simply taking  $\tilde{K}_2=0$ and $K_2=K_1$. We therefore provide just the proof of statements (ii) and (iii) when $q>2$. 
\begin{footnote}{One could actually prove just statement (iii) and then derive statements (i) and (ii) (see Remark \ref{rem_cases}). We prefer to provide separate proofs for statements (ii) and (iii) to emphasize how the condition \eqref{visco1} appears in statement (iii).}\end{footnote}%The proof of \eqref{E_sup_q} can be found in \cite[Appendix 1]{Fla_Gat}. 
\begin{itemize}
\item [(ii)]
Let $u$ be the solution to the Navier-Stokes equation \eqref{NS_abs}.
We apply the It\^o formula to the functional $\|u(t)\|^q_H$, $q > 2$. 
Exploiting \eqref{B=0} and bearing in mind the previous computation for $q=2$, we infer 
\begin{align}
\label{sti3}
{\rm d}\|u(t)\|^q_H + q \nu \|u(t)\|_H^{q-2}\|u(t)\|^2_V\, {\rm d}t
&\le \frac{q(q-1)}{2}\|u(t)\|^{q-2}_H\|G(u(t))\|^2_{L_{HS}(U,H)}\, {\rm d}t 
\notag\\
& \qquad +q\|u(t)\|^{q-2}_H\langle u(t), G(u(t)){\rm d}W(t)\rangle + q\|u(t)\|_H^{q-2}\langle u(t), f\rangle\,{\rm d}t.
\end{align}
Using repeatedly the Young inequality we get that for any $\varepsilon>0$  and $\eta>0$ there exists a constant 
$C_1=C_1(\varepsilon, \eta, q, \nu)$ such that
\begin{align}
\label{sti_f_2}
q\|u\|_H^{q-2}\langle u, f\rangle
\le 
 q\|u\|_H^{q-2} \left(\eta  \nu   \|u\|^2_V + \frac{1}{\eta  \nu  }\|f\|^2_{V^*} \right)
\le 
 q\eta  \nu  \|u\|^{q-2}_H\|u\|^2_V+ \frac{\varepsilon}{2}\|u\|_H^q + C_1\|f\|_{V^*}^q. 
% \le \frac{q\nu}{2}\|u\|_H^{q-2}\|u\|^2_V+ \frac{q}{2\nu}\|u\|_H^{q-2}\|f\|^2_{V^*} \le  \frac{q\nu}{2}\|u\|_H^{q-2}\|u\|^2_V +\frac{\varepsilon}{2} \|u\|_H^q+C \|f\|_{V^*}^q.
\end{align}
From (G2)(ii) and the Young inequality, for any  $\varepsilon>0$ there  exists a constant 
$C_2=C_2(\varepsilon, \gamma,q, K_2, \tilde{K}_2)$ such that
\begin{equation}
\label{sti4}
\frac{q(q-1)}{2}\|u\|^{q-2}_H\|G(u)\|^2_{L_{HS}(U,H)} \le C_2+ \frac{\varepsilon}{2} \|u\|_H^q.
\end{equation}
We choose $\eta=\frac 12$ and 
insert \eqref{sti_f_2} and \eqref{sti4} into \eqref{sti3}; we get
\begin{equation}
\label{sti3bis}
\begin{split}
{\rm d}\|u(t)\|^q_H &+ q \frac \nu 2 \|u(t)\|^{q-2}_H\|u(t)\|^2_V\, {\rm d}t- \varepsilon \|u(t)\|^q_H
\\
&\le C_2 + C_1\|f\|_{V^*}^q +q\|u(t)\|^{q-2}_H\langle u(t), G(u(t)){\rm d}W(t)\rangle.
\end{split}
\end{equation}
Since the stochastic integral is a martingale (thanks to \eqref{q_mom_sol}), 
taking the expected value in both sides of \eqref{sti3bis} and exploiting the Poincar\'e inequality  \eqref{lambda_1}, we find
\begin{align*}
\frac{{\rm d}}{{\rm d}t}\mathbb{E}\left[\|u(t)\|^q_H \right]\le C_2 + C_1\|f\|_{V^*}^q-\left(q \frac \nu 2 \lambda_1-\varepsilon\right)\mathbb{E}\left[\|u(t)\|^q_H\right].
 \end{align*}
 Choosing  $\varepsilon\le \frac{q \nu \lambda_1}{4}$ we get
 \[
 \frac{\rm d}{{\rm d}t} \mathbb{E}\left[\|u(t)\|^q_H \right]
 \le 
 -\frac{q \nu \lambda_1}{4} \mathbb{E}\left[\|u(t)\|^q_H \right]
+C_3
 \]
 where $C_3= C_3(\nu, \gamma, q, K_2, \tilde{K}_2, \|f\|_{V^*}) $. By Gronwall lemma we obtain
\[
\mathbb{E}\left[ \|u(t)\|^q_H\right]\le \|u_0\|^q_H e^{-\frac{q \nu \lambda_1}4 t}+ \frac4{q \nu \lambda_1} C_3
\]
and the thesis follows with $\bar C= \frac{q \nu \lambda_1}4$ and $\underline C=\frac4{q \nu \lambda_1} C_3$.  
 \item [(iii)]
 Notice at first that the condition on the viscosity coefficient ensures to have a non-empty set of admissible parameters $q$.
The proof follows then the lines of case (ii): we still have estimates \eqref{sti3} and \eqref{sti_f_2} but now, by means of the Young inequality, for  any arbitrary $\varepsilon>0$ we estimate as in \eqref{G(u)_est} and get
\begin{equation}
\label{sti5}
\frac{q(q-1)}{2}\|u(t)\|^{q-2}_H\|G(u)\|^2_{L_{HS}(U,H)} \le  C_4 + \frac{q(q-1)}{2}(1+\varepsilon)\tilde{K}_3^2 \|u(t)\|^q_H,
\end{equation}
with $C_4=C_4(\varepsilon, K_3,q)$.
Using \eqref{lambda_1}, \eqref{sti_f_2}, \eqref{sti5} and the fact that the stochastic integral is a martingale, by taking the expected value on both sides of \eqref{sti3}, we obtain for $C_1=C_1(\varepsilon, \eta, \nu, q)$
the same constant as above,
\begin{equation}
%\label{sti2}
\frac{{\rm d}}{{\rm d}t}\mathbb{E}\left[\|u(t)\|_H^q\right] 
+\left(q\nu(1-\eta) \lambda_1
- \frac  \varepsilon 2 - \frac {1+\varepsilon}2 q(q-1)\tilde{K}_3^2\right) \mathbb{E}\left[\|u(t)\|^q_H\right] 
\le 
C_4+C_1\|f\|_{V^*}^q.
 \end{equation} 
Thanks to  assumption \eqref{visco1},  we can find $\varepsilon=\varepsilon(q,\nu,\lambda_1, \tilde{K}_3)>0$ and $\eta=\eta(q, \nu, \lambda_1)>0$ small enough such that $\bar C:=q\nu(1-\eta) \lambda_1
- \frac  \varepsilon 2 - \frac {1+\varepsilon}2q(q-1)\tilde{K}_3^2>0$
%and we obtain\begin{equation*}\frac{{\rm d}}{{\rm d}t} \mathbb{E}\left[ \|u(t)\|^q_H\right] \le C(\varepsilon, K_1^2)-K\mathbb{E}\left[ \|u(t)\|^q_H\right]\end{equation*}
and the Gronwall lemma yields
\begin{equation*}
\mathbb{E}\left[ \|u(t)\|^q_H\right]\le \|u_0\|^q_H e^{-\bar C t}+ \frac{C_4+C_1\|f\|_{V^*}^q}{\bar C}\left(1-e^{-\bar Ct}\right),
\end{equation*}
and the thesis follows by taking $\underline{C}=\frac{C_4+C_1\|f\|_{V^*}^q}{\bar C}$.
\end{itemize}
\end{proof}

We now provide an a priori estimate on the solution to  the nudged equation \eqref{NS_abs_nud} with $u_0,v_0\in H$.

\begin{lemma}
\label{lem_v}
Assume (G1). 
\begin{itemize}
\item [(i)] If (G2)(i) holds, then for any $q \ge 2$ the solution $v=v(v_0, u_0)$ to the nudged equation \eqref{NS_abs_nud} satisfies 
\begin{equation}
\label{est_v}
\sup_{t \ge 0}\mathbb{E}\left[\|v(t)\|^q_H\right] \le C(1+ \|u_0\|^q_H+ \|v_0\|^q_H),
\end{equation} 
where $C=C(q,\nu, K_1, \lambda_1, \|f\|_{V^*})$.
\item [(ii)] If (G2)(ii) holds, then for any $q \ge 2$ the solution $v$ to the nudged equation \eqref{NS_abs_nud} satisfies  \eqref{est_v} with $C=C(q, \nu, K_2, \tilde{K}_2, \lambda_1, \gamma, \|f\|_{V^*})$.
\item[(iii)] If (G2)(iii) and  \eqref{visco1} hold, then  for any $q\in [2,  1+\frac{2\nu \lambda_1}{\tilde{K}_3^2})$
the solution $v$ to the nudged equation \eqref{NS_abs_nud} satisfies \eqref{est_v} with $C=C(q,\nu, K_3, \tilde{K}_3, \lambda_1, \|f\|_{V^*})$ .
\end{itemize}
\end{lemma}
%PROOF
\begin{proof}
Let $v=v(v_0, u_0)$ be the solution to \eqref{NS_abs_nud}. Let $q=2$. We apply the It\^o formula to the functional $\|v(t)\|^2_H$. 
Exploiting \eqref{B=0} we infer, $\mathbb{P}$-a.s., for any $t\ge0$,
\begin{align*}
%\label{sti1}
{\rm d}\|v(t)\|^2_H + 2 \nu \|v(t)\|^2_V\, {\rm d}t  = 
[\|G(v(t))\|^2_{L_{HS}(U,H)}  + 2 \langle v(t), f\rangle+ 2 \lambda\langle P_N(u(t)-v(t)), v(t)\rangle]\, {\rm d}t
+2 \langle v(t), G(v(t))\, {\rm d}W(t)\rangle.
\end{align*}
By means of the Cauchy-Schwartz and the Young inequalities, we estimate
\begin{equation*}\langle P_N(u-v), v\rangle \le \|u\|_H\|v\|_H-\|P_Nv\|^2_H \le \frac {1}{2\varepsilon} \|u\|^2_H + \frac {\varepsilon}{2} \|v\|^2_H,
\end{equation*}
for  any $\varepsilon>0$.

Proceeding as in Lemma \ref{lem_A1} and using  the Poincar\'e inequality  \eqref{lambda_1}, we infer 
\[
{\rm d}\|v(t)\|^2_H + a\lambda_1 \|v(t)\|_H^2\, {\rm d}t \le  \left[b + \frac{\lambda}{\varepsilon} \|u(t)\|_H^2 + \varepsilon \lambda \|v(t)\|^2_H \right]\, {\rm d}t + 2\langle v(t), G(v(t))\, {\rm d}W(t)\rangle,
\]
with $a$ and $b$ as in \eqref{a} and \eqref{b}, respectively. 
We choose $\varepsilon$ small enough  so that
$a\lambda_1-\varepsilon \lambda\ge \frac{a\lambda_1}{2}=:\bar a$.
Hence there exists a  positive constant $C_6$ such that 
\begin{equation}
\label{sti_v_1}
{\rm d}\|v(t)\|^2_H + \bar a \|v(t)\|_H^2\, {\rm d}t \le C_6 \left[1 + \|u(t)\|_H^2 \right]\, {\rm d}t + 2\langle v(t), G(v(t))\, {\rm d}W(t)\rangle,
\end{equation}
We now take the expected value on both sides of \eqref{sti_v_1}. Using the fact that the stochastic term is a martingale (thanks to \eqref{q_mom_sol}), exploiting the estimate $\sup_{t \ge 0}\mathbb{E}\left[\|u(t)\|^2_H\right] \le C(1+ \|u_0\|_H^2)$ that follows from estimate \eqref{sup_q_est} in Lemma \ref{lem_A2} we infer
\begin{equation*}
\frac{{\rm d}}{{\rm d}t}\mathbb{E}\left[\|v(t)\|_H^2\right] \le - \bar a \mathbb{E}\left[\|v(t)\|_H^2\right] +  C(1+ \|u_0\|^2_H).
\end{equation*}
The Gronwall lemma then yields
\begin{equation*}
\mathbb{E}\left[\|v(t)\|_H^2 \right] \le \|v_0\|^2_H e^{-\bar a t}+  \frac{C}{\bar a} (1+ \|u_0\|_H^2),
\qquad \forall t\ge 0
\end{equation*}
from which  the thesis follows.

For $q>2$ the proof follows the lines of the proof of Lemma \ref{lem_A2}. We apply the It\^o formula to the functional $\|v(t)\|^q_H$ and  obtain an equation  for $v$ which is of the form \eqref{sti3} where now it also appears the additional term 
\begin{equation}
\lambda q \|v(t)\|_H^{q-2}\langle v(t), P_N(u(t)-v(t))\rangle.
\end{equation}
By means of the Cauchy-Schwartz and the Young inequalities, we estimate, for any $\delta>0$,
\begin{align*}
\lambda q \|v\|_H^{q-2}\langle v, P_N(u-v)\rangle
\le \lambda q \|v\|_H^{q-2}\left( \frac{1}{2\delta}\|u\|_H^2+\frac{\delta}{2}\|v\|_H^2\right)\le \lambda q\delta\|v\|_H^q + C(q, \lambda, \delta)\|u\|_H^q.
\end{align*}
Bearing in mind the above estimate and arguing  as in the proof of Lemma \ref{lem_A2} the thesis follows.
\end{proof}

\subsection{Estimates in probability}
According to the different assumptions (G2)(i), (ii) or (iii) that we impose on the operator $G$,  we have different estimates in probability for the solution of the Navier-Stokes equation \eqref{NS_abs}. 
We collect them in the following two Propositions.

\begin{proposition}
\label{lem_exp}
Assume (G1) and (G2)(i). Let $u$ denote the corresponding solution of  the Navier-Stokes equation \eqref{NS_abs}. Then
\begin{equation}
\label{P_est_exp}
\mathbb{P}\left( \sup_{t\ge 0}\left[ \|u(t)\|^2_H + \frac{\nu}{2} \int_0^t \|u(s)\|^2_V \, {\rm d}s -\|u_0\|^2_H -\left( K_1^2+ \frac{\|f\|^2_{V^*}}{\nu}\right)t\right] \ge R\right) \le e^{-\frac{\nu\lambda_1}{8K_1^2}R},
\end{equation}
for all $R>0$.
\end{proposition}
%PROOF
\begin{proof}
From estimates \eqref{sti_gen}, \eqref{a} and \eqref{b} we get, $\mathbb{P}$-a.s., for all $t\ge0$, 
\begin{align}
%\label{sti_G_bou}
\|u(t)\|^2_H +\nu\int_0^t \|u(s)\|^2_V\, {\rm d}s  \le  \|u_0\|_H^2+ \left(K_1^2+ \frac{\|f\|^2_{V^*}}{\nu} \right) t + M(t),
\end{align}
where $M$ is defined in \eqref{M} and its quadratic variation is estimated as
\begin{equation}
%\label{qvM}
[M](t) \le %\int_0^t\|u(s)\|^2_H \|G(u(s))\|^2_{L_{HS}(U,H)}\le 
4K_1^2\int_0^t \|u(s)\|_H^2\,{\rm d}s
\underset{\text{by  }\eqref{lambda_1}} {\le}
\frac{4K_1^2}{\lambda_1} \int_0^t \|u(s)\|_V^2\,{\rm d}s.
\end{equation}
Hence
\[
\|u(t)\|^2_H +\frac{\nu}{2}\int_0^t \|u(s)\|^2_V\, {\rm d}s  -  \|u_0\|_H^2- \left(K_1^2+ \frac{\|f\|^2_{V^*}}{\nu} \right) t 
\le 
M(t) -\frac{\nu}{2}\int_0^t \|u(s)\|^2_V\, {\rm d}s
\le 
M(t) - \frac{\nu\lambda_1}{8K_1^2}[M](t).
\]
The thesis is  obtained from the exponential martingale inequality
\begin{equation*}
\mathbb{P}\left( \sup_{t \ge 0} \left[M(t)-\alpha [M](t)\right] \ge R\right) \le e^{-\alpha R}, \quad \forall \ R, \alpha>0
\end{equation*}
with $\alpha= \frac{\nu\lambda_1}{8K_1^2}$.
\end{proof}

%%% PROP  %%%%
\begin{proposition}
\label{lem_pol}
Assume (G1). 
Let $u$ denote the solution of  the Navier-Stokes equation \eqref{NS_abs}. 
Set $C_b=\min(1+b,2)$, where $b$ is defined in \eqref{b}.
\begin{enumerate}%[label=$(\roman{*})$]
\item[1.]  If  (G2)(ii) holds, then there exists a positive constant $C=C(\lambda_1,q,\nu, K_2, \tilde{K_2},\gamma, \|f\|_{V^*})$ such
 that for any arbitrary $q> 2$ 
\begin{align}
\label{P_est_ii}
\mathbb{P}\left( \sup_{t \ge T} \left[ \|u(t)\|^2_H+\nu \int_0^t \|u(s)\|^2_V\, {\rm d}s -\|u_0\|^2_H-C_b (t+1)\right] \ge R\right)
\le \frac{C(1+ \|u_0\|_H^{2q})}{(T+R)^{\frac q2-1}},
\end{align}
for all $T\ge 0$, $R>0$.
\item[2.]
If  (G2)(iii) holds and 
\begin{equation}\tag{\ref{condizione32}}
\nu > \frac{3\tilde{K}_3^2}{2\lambda_1}, 
\end{equation}
then there exists a positive constant $C=C(\lambda_1,q,\nu, K_3, \tilde{K_3}, \|f\|_{V^*})$ such
 that for any arbitrary $q\in \left(2,\frac 12 + \frac{\nu \lambda_1}{\tilde{K}_3^2}\right)$  
\begin{align}
\label{P_est_iii}
\mathbb{P}\left( \sup_{t \ge T} \left[ \|u(t)\|^2_H+ (\nu-\tfrac {\tilde K_3^2}{2\lambda_1})\int_0^t \|u(s)\|^2_V\, {\rm d}s -\|u_0\|^2_H-C_b (t+1)\right] \ge R\right)
\le \frac{C(1+ \|u_0\|_H^{2q})}{(T+R)^{\frac q2-1}},\end{align}
for all $T\ge 0$, $R>0$.
\end{enumerate}
\end{proposition}
%PROOF
\begin{proof}
When $G$ is unbounded, we proceed differently than in the bounded case, since the quadratic variation of the stochastic integral in the It\^o formula \eqref{sti1} has a growth with a power larger than 2 and thus cannot be balanced by the integral $\int_0^t \|u(s)\|^2_V\, {\rm d}s$ appearing in the l.h.s.

We start from estimate  \eqref{sti_gen} with $a$, $b$ as in \eqref{a}, \eqref{b} respectively and set 
$C_b = \min(b+1,2)$.
Therefore for any $R,T>0$,  we have
\begin{align}
\label{Mart1}
\mathbb{P} \left( \sup_{t \ge T}\left[ \|u(t)\|^2_H+a\int_0^t \|u(s)\|^2_V\, {\rm d}s -\|u_0\|^2_H-C_b(t+1)\right]\ge R\right) 
\le 
\mathbb{P}\left( \sup_{t \ge T} \left[ M(t)-t-2\right]\ge R\right).
\end{align}

We observe that, for any $T\ge 0$, $R > 0$,
\begin{equation}
\label{Mart2}
\left\{\sup_{t \ge T}\left[  M(t)-t-2\right] \ge R \right\} \subset \bigcup_{m \ge \lfloor T\rfloor}\left\{ \sup_{t \in [m, m+1)}\left[ M(t)-t-2\right] \ge R\right\},
\end{equation}
where  $\lfloor T\rfloor$ denotes the largest integer less than or equal to $T$. On the other hand, notice that for $R>0$ and any $m \ge 0$ 
\begin{equation}
\label{Mart3}
\left\{ \sup_{t \in [m, m+1)}\left[ M(t)-t-2\right] \ge R\right\} \subset \left\{ M^*(m+1) \ge R+m+2\right\},
\end{equation} 
where we adopt the notation $M^*(t):= \sup_{s \in [0,t] }| M(s)|$.
We will exploit the Burkholder-Davis-Gundy 
\[
\mathbb{E}\left[M^*(t)^q\right] \lesssim_q  \mathbb{E}\left[ \left[M \right](t)^{\frac q2} 
\right]
\]
 in order to obtain a suitable estimate for \eqref{Mart1} from \eqref{Mart2} and \eqref{Mart3}. 

By means of the Young inequality, under either (G2)(ii) or (G2)(iii), we can estimate the quadratic variation $[M](t)$ as follows
\begin{align*}
[M](t)
%&=4\int_0^t\sum_{k=1}^\infty|\langle u(s), G(u(s))f_k\rangle|^2\,{\rm d}s
&\le 4  \int_0^t \|u(s)\|^2_H\|G(u(s))\|^2_{L_{HS}(U,H)}\, {\rm d}s
\le C_1\int_0^t \left(1+ \|u(s)\|^4_H \right)\, {\rm d}s,
\end{align*}
where 
\begin{equation*}
C_1
=
\begin{cases}
C_1(K_2, \tilde{K_2}, \gamma) & \text{under (G2)(ii)}
\\
C_1(K_3, \tilde K_3) & \text{under (G2)(iii)}
\end{cases}
\end{equation*}
 is a positive constant (see \eqref{G(u)_est}).
\\
Thus, from the Burkholder-Davis-Gundy and the H\"older inequalities and \eqref{sup_q_est}
%\begin{footnote}{\textcolor{red}{Qua mi serve l'assunzione: $\nu >\frac{(2q-1)\tilde{K_1}^2}{2\lambda_1}$. Si pu\'o indebolire? Non si pu\'o usare \eqref{E_sup_q} pwech\'e si ha una crescita esponenziale in tempo.}}\end{footnote} 
we find that for all $q \ge 2$,
\begin{align}
\label{est_quadr_var}
\mathbb{E} \left[M^*(t) \right]^q
&\lesssim_q\mathbb{E}\left[ [M](t)^{\frac q2}\right]
\lesssim_{q, C_1} \mathbb{E}\left[\left(\int_0^t  \left(1+\|u(s)\|^{4}_H\right)\, {\rm d}s\right)^{\frac q2}\right]
\\\notag
&\lesssim_{q,C_1}t^{\frac{q-2}{2}} \mathbb{E}\left[\int_0^t  \left(1+\|u(s)\|^{2q}_H\right)\, {\rm d}s\right]
%\le C(t+1)^{\frac{q-2}{2}}\left((t+1)+\|u_0\|^{2q}_H \right).
\le C(t+1)^{\frac{q}{2}}\left(1+\|u_0\|^{2q}_H \right),
\end{align} 
where 
\begin{equation}
\label{C}
C=
\begin{cases}
C(K_2, \tilde K_2, \nu, \lambda_1, q, \gamma, \|f\|_{V^*}) & \text{under (G2)(ii)}
\\
C(K_3, \tilde K_3, \nu, \lambda_1, q, \|f\|_{V^*}) & \text{under (G2)(iii)},\end{cases}
\end{equation}
is a positive constant.
We have to require $2q< 1+\frac{ 2 \nu\lambda_1}{\tilde{K}_3^2}$
 in order to use \eqref{sup_q_est} when Assumption (G2)(iii) is in force.
 
From \eqref{Mart2}, \eqref{Mart3}, the Chebychev inequality and \eqref{est_quadr_var}, where the constant $C$ is as in \eqref{C}, we have
\begin{multline*}
\mathbb{P}\left(\sup_{t \ge T}\left[ M(t)-t-2\right] \ge R \right) \le \sum_{m \ge  \lfloor T\rfloor}\mathbb{P}\Big( M^*(m+1) \ge R+m+2\Big)
\le \sum_{m \ge  \lfloor T\rfloor}\frac{\mathbb{E}\left[M^*(m+1)^q\right]} {(R+m+2)^q}
\\
\le C(1+ \|u_0\|_H^{2q})\sum_{m \ge  \lfloor T\rfloor}\frac{(m+2)^{\frac q2}} {(R+m+2)^q}
\le C(1+ \|u_0\|_H^{2q})\sum_{m \ge  \lfloor T\rfloor}\frac{1} {(R+m+2)^{\frac q2}}
\end{multline*}
The latter series is convergent when $q > 2$ and thus we obtain
\[
\mathbb{P}\left(\sup_{t \ge T}\left[M(t)-t-2\right] \ge R \right)
\lesssim
  \frac{1+ \|u_0\|_H^{2q}}{(T+R)^{\frac q2-1}}.
\]
Under Assumption (G2)(iii) the condition $q>2$ requires that $2<\frac 12 +\frac{\nu\lambda_1}{\tilde{K}_3^2} $, which is \eqref{condizione32}.
Keeping in mind \eqref{Mart1},
the estimates \eqref{P_est_ii} and \eqref{P_est_iii} follow.
\end{proof}

\section{Proof of Lemma \ref{techn_lemm}.}
\label{dim_lem}
\begin{proof}
Set $r:=u_1-u_2$; this difference satisfies
\begin{equation*}
{\rm d}r(t)+ \left[ \nu Ar(t)+B(r(t),u_1(t))+B(u_2(t),r(t))\right]\,{\rm d}t=(G(u_1(t))-G(u_2(t)))\,{\rm d}W(t)
\end{equation*}
with $r(0)=x-y$.
We follow an idea of \cite{Sch1997} and we apply the It\^o formula to ${\rm d}\left( e^{-\int_0^t \psi(s)\, {\rm d}s}\|r(t)\|^2_H\right)$, choosing $\psi$ as
\begin{equation*}
\psi(t):= L_G^2-\lambda_1 \nu+\frac{1}{\nu} \|u_1(t)\|^2_V,
\end{equation*}
where we recall that $L_G$ is the constant appearing in Assumption (G1) and $\lambda_1$ is the first eigenvalue of the Laplace operator.
We recall that $u_1 \in L^2(0,T,V)$ $\widetilde{\mathbb{P}}$-a.s., so $\psi \in L^1(0,T)$ $\widetilde{\mathbb{P}}$-a.s..
We have 
\begin{equation*}
{\rm d}\left( e^{-\int_0^t \psi(s)\, {\rm d}s}\|r(t)\|^2_H\right)
=
-\psi(t)e^{-\int_0^t \psi(s)\, {\rm d}s}\|r(t)\|^2_H +e^{-\int_0^t \psi(s)\, {\rm d}s}\, {\rm d}\|r(t)\|^2_H.
\end{equation*}
By similar computations as the ones done in the proof of Theorem \ref{FP_thm} one obtains 
\begin{equation*}
{\rm d}\|r(t)\|^2_H \le \left( L_G^2-\lambda_1 \nu + \frac{1}{\nu}\|u_1(t)\|^2_V\right) \|r(t)\|_H^2 + \langle G(u_1(t))-G(u_2(s)), r(t)\, {\rm d}W(t)\rangle.
\end{equation*}
Thus 
\begin{equation}
\label{sti_techn_1}
{\rm d}\left( e^{-\int_0^t \psi(s)\, {\rm d}s}\|r(t)\|^2_H\right)
\le e^{-\int_0^t \psi(s)\, {\rm d}s} \langle G(u_1(t))-G(u_2(s)), r(t)\, {\rm d}W(t)\rangle.\end{equation}
The r.h.s. is a martingale, in fact define
\begin{equation*}
N(t):=e^{-\int_0^t \psi(s)\, {\rm d}s} \langle G(u_1(t))-G(u_2(s)), r(t)\, {\rm d}W(t)\rangle.
\end{equation*}
Then, Assumption (G1) yields 
\begin{align*}
\widetilde{\mathbb{E}}[N(t)^2]
&\le \widetilde{\mathbb{E}}\left[ \int_0^t e^{-2\int_0^t \psi(s)\, {\rm d}s}\|r(s)\|^2_H\|G(u_1(s))-G(u_2(s))\|^2_{L_{HS}(U,H)}\, {\rm d} s\right]
\\
&\le L_G^2e^{2\lambda_1 \nu}\widetilde{\mathbb{E}}\left[\int_0^t\|r(s)\|^4_H\, {\rm d} s\right]
\end{align*}
which is finite thanks to \eqref{q_mom_sol}.
Therefore, by integrating \eqref{sti_techn_1} over $[0,t]$ and taking the expected value on both sides, we get
\begin{equation*}
\widetilde{\mathbb{E}}\left[e^{-\int_0^t \psi(s)\, {\rm d}s}\|r(t)\|^2_H \right] \le \|x-y\|^2_H
\end{equation*}
and this concludes the proof.
\end{proof}

%%%%%%%%%%%%%%%%%%%%%%%%%%%%%%%%%%%%%%%%%%%%%%%%%%%%%%%%%%%%%%%%%%%%%%%%%%%%%%%%%%%%%%%%%%%%%%%%%%%%%%%%%%%%%%%%%%%%%%%%%%%%%%

%\subsection{Polynomial vs exponential convegence}Se avessimo un noise bounded mi sembra potremmo usare l'exponential martingal inequality per controllare $\tau_{R, \beta}$: di fatto, con un. noise bounded le stime sono analoghe ad un noise additivo: guarda ad esempio Lemma 2.3 in Kuksin-Shyrikian2002. Se per\'o consideriamo un noise moltiplicativo ma NOT bounded (e.g. a crescita sublineare), non mi sembra che si riesca ad usare l'exponential martingale inequality: bisogna passare per le stime che danno decadimento polinomiale. QUESTO \'E ESATTAMENTE QUELLO CHE DICE ODASSO 2008 NEL REMARK 3.7!

\section*{Acknowledgements}

The authors are members of Gruppo Nazionale per l’Analisi Matematica, la Probabilità e le loro Applicazioni (GNAMPA) of the Istituto Nazionale di Alta Matematica (INdAM), and gratefully acknowledge financial support through the project %CUP$-$E55F22000270001 %(Scarpa) 
CUP$-$E53C22001930001. %(Zanella)

\addcontentsline{toc}{section}{\numberline{} Bibliography} 
\bibliographystyle{abbrv}

\end{document}